\documentclass[12pt]{amsart}
\usepackage{amsmath,amssymb,latexsym,soul,cite,amsthm,color,enumitem,graphicx,mathtools,microtype,tikz}
\usepackage[colorlinks=true,urlcolor=cobalt,citecolor=cobalt,linkcolor=cobalt,linktocpage,pdfpagelabels,bookmarksnumbered,bookmarksopen]{hyperref}
\definecolor{cobalt}{rgb}{0.0, 0.28, 0.67}
\usepackage[english]{babel}
\usepackage[left=2.7cm,right=2.7cm,top=2.4cm,bottom=2.5cm]{geometry}
\numberwithin{equation}{section}

\newtheorem{theorem}{Theorem}[section]
\theoremstyle{plain}
\newtheorem{lemma}[theorem]{Lemma}
\theoremstyle{plain}
\newtheorem{proposition}[theorem]{Proposition}
\theoremstyle{plain}

\theoremstyle{definition}

\newtheorem{example}[theorem]{Example}

\newcommand{\N}{{\mathbb N}}

\newcommand{\R}{{\mathbb R}}
\newcommand{\eps}{\varepsilon}
\newcommand{\beq}{\begin{equation}}
\newcommand{\eeq}{\end{equation}}
\renewcommand{\le}{\leqslant}
\renewcommand{\ge}{\geqslant}

\newcommand{\w}{W^{s,p}_0(\Omega)}
\newcommand{\fpl}{(-\Delta)_p^s\,}
\newcommand{\ds}{{\rm d}_\Omega^s}

\newenvironment{enumroman}{\begin{enumerate}

}{\end{enumerate}}

\def\Xint#1{\mathchoice
{\XXint\displaystyle\textstyle{#1}}%
{\XXint\textstyle\scriptstyle{#1}}%
{\XXint\scriptstyle\scriptscriptstyle{#1}}%
{\XXint\scriptscriptstyle\scriptscriptstyle{#1}}%
\!\int}
\def\XXint#1#2#3{{\setbox0=\hbox{$#1{#2#3}{\int}$ }
\vcenter{\hbox{$#2#3$ }}\kern-.6\wd0}}
\def\dashint{\Xint-}

\title[Singular fractional $p$-Laplacian]{Fine boundary regularity for the singular fractional $p$-Laplacian}

\author[A.\ Iannizzotto, S.\ Mosconi]{Antonio Iannizzotto, Sunra Mosconi}

\address[A.\ Iannizzotto]{Department of Mathematics and Computer Science
\newline\indent
University of Cagliari
\newline\indent
Via Ospedale 72, 09124 Cagliari, Italy}
\email{antonio.iannizzotto@unica.it}

\address[S.\ Mosconi]{Department of Mathematics and Computer Science
\newline\indent
University of Catania
\newline\indent
Viale A.\ Doria 6, 95125 Catania, Italy}
\email{sunra.mosconi@unict.it}


\begin{document}

\begin{abstract}
We study the boundary weighted regularity of weak solutions $u$ to a $s$-fractional $p$-Laplacian equation in a bounded smooth domain $\Omega$ with bounded reaction and nonlocal Dirichlet type boundary condition, in the singular case $p\in(1,2)$ and with $s\in(0,1)$. We prove that $u/\ds$ has a $\alpha$-H\"older continuous extension to the closure of $\Omega$, ${\rm d}_\Omega(x)$ meaning the distance of $x$ from the complement of $\Omega$. This result corresponds to that of \cite{IMS1} for the degenerate case $p\ge 2$.
\end{abstract}

\maketitle

\begin{center}
\small
Version of \today\
\vskip4pt
\begin{minipage}{9cm}
\tableofcontents
\end{minipage}
\end{center}
	
\section{Introduction}\label{sec1}
 
\subsection{Main result}
In the present paper we study a form of fine boundary regularity for nonlinear, nonlocal elliptic equations of fractional order, coupled with a Dirichlet condition. Precisely, we consider the following nonlocal Dirichlet type problem:
\beq\label{dir}
\begin{cases}
\fpl u = f(x) & \text{in $\Omega$} \\
u = 0 & \text{in $\R^N\setminus\Omega$.}
\end{cases}
\eeq
Here $\Omega\subset\R^N$  is a bounded domain with a $C^{1,1}$-smooth boundary $\partial\Omega$, $p>1$, $s\in(0,1)$, and the leading operator is the $s$-fractional $p$-Laplacian, defined for any $u$ in the fractional Sobolev space $W^{s,p}(\R^N)$ as the gradient of the functional
\[u \mapsto \frac{1}{p}\iint_{\R^N\times\R^N}\frac{|u(x)-u(y)|^p}{|x-y|^{N+ps}}\,dx\,dy.\]
Also, the reaction is a function $f\in L^\infty(\Omega)$, and the Dirichlet condition prescribes vanishing of $u$ a.e.\ in $\R^N\setminus\Omega$. By classical variational arguments, problem \eqref{dir} admits a unique weak solution $u$ lying in a convenient fractional Sobolev space $W^{s,p}_0(\Omega)$ incorporating the Dirichlet condition. Such solution is H\"older continuous in $\overline\Omega$ (see \cite{IMS}) and nothing more in general, so we are interested in a form of fine (or weighted) H\"older regularity involving the distance function
\[{\rm d}_\Omega(x) = {\rm dist}(x,\R^N\setminus\Omega).\]
Our result is the following:

\begin{theorem}\label{main}
Let $p>1$, $s\in(0,1)$, $\Omega\subset\R^N$ be a bounded domain with a $C^{1,1}$-smooth boundary $\partial\Omega$. Then, there exist $\alpha\in(0,s)$, $C>0$, depending on $N$, $p$, $s$, and $\Omega$, with the following property: for all $f\in L^\infty(\Omega)$, if $u\in W^{s,p}_0(\Omega)$ is the weak solution of problem \eqref{dir}, then $u/\ds$ admits a $\alpha$-H\"older continuous extension to $\overline\Omega$ and it satisfies the uniform bound
\[\Big\|\frac{u}{\ds}\Big\|_{C^\alpha(\overline\Omega)} \le C\|f\|_{L^\infty(\Omega)}^\frac{1}{p-1}.\]
\end{theorem}

\subsection{Related results}
In order to fully understand the meaning of Theorem \ref{main}, we will now draw a brief {\em r\'esum\'e} of some relevant regularity results for nonlocal elliptic operators. First, let us consider the equation
\beq\label{leq}
Lu = f(x) \quad \text{in $\Omega$,}
\eeq
where $L$ denotes a {\em linear} elliptic operator with fractional order of differentiation $2s$ ($s\in(0,1)$), including the model case of the fractional Laplacian $L=(-\Delta)^s$. Regularity of the solutions of \eqref{leq} is well understood. In the model case $L=(-\Delta)^s$, Schauder estimates follow from standard potential theory and ensure:
\begin{itemize}[leftmargin=1cm]
\item $u\in C^{2s+\alpha}(\Omega)$ as long as $f\in C^\alpha(\Omega)$ and $2s+\alpha\notin\N$;
\item  $u\in C^{2s}(\Omega)$ if $f\in L^\infty(\Omega)$ (except for $s=1/2$, in which case $u\in C^{2s-\eps}(\Omega)$ $\forall\eps>0$).
 \end{itemize}
A similar result holds for far more general fractional linear operators $L$ which are {\em translation invariant}, meaning $Lu (\cdot+z)=Lu$, arising as infinitesimal generators of $2s$-stable L\'evy processess (see  \cite{BCI,DRSV,ROS2} and also \cite{MY,FRO} for the regional fractional Laplacian, corresponding to a censored L\'evy process). When the linear fractional operator $L$ is not translation invariant due to the presence of coefficients, one has a corresponding notion of divergence versus non-divergence form of the equation.  If no assumption is made on the coefficients beyond boundedness and measurability, the best one can expect is $\alpha$-H\"older regularity in the interior with a small, not explicit $\alpha$ (see \cite{DK,K,M} for the divergence case and \cite{S1} for the non-divergence case). When the coefficients are assumed to be $\alpha$-H\"older continuous, the initial Schauder type interior regularity statements holds true in the non-divergence case, see \cite{B,F,FRO1}, and have to be naturally modified for divergence form operators \cite{FRO1}.
\vskip2pt
\noindent
Now let us couple equation \eqref{leq} with a nonlocal, homogeneous Dirichlet condition:
\beq\label{fed}
\begin{cases}
Lu=f(x) & \text{in $\Omega$} \\
u= 0 & \text{in $\R^N\setminus\Omega$,}
\end{cases}
\eeq
where, in the following discussion, $\partial\Omega$ is assumed to be smooth. The regularity up to the boundary for problem \eqref{fed} differs substantially from the interior one, as is clear observing that the function $u(x)= (x_+)^s$ solves $(-\Delta)^s u=0$ in $(0,\infty)$. For the fractional Laplacian the optimal global regularity is $u\in C^s(\R^N)$ and the same holds true for the translation invariant fractional operators discussed above. One is then led to study the {\em fine boundary regularity} of $u$, i.e., the boundary regularity of $u/\ds$. If $f\in L^\infty$ in \eqref{fed}, then $u/\ds\in C^{s-\eps}(\overline{\Omega})$ for $\eps>0$ arbitrarily small (see \cite{ROS2}). For more regular $f$, the corresponding Schauder theory is developed in \cite{ARO,AG}. Fine boundary regularity for different classes of linear elliptic fractional operators involve  $u/{\rm d}_{\Omega}^\beta$ for $\beta\neq s$, see for instance  \cite{F1,FRO,DRSV}. 
\vskip2pt
\noindent
Fully nonlinear, uniformly elliptic operators of fractional order have been studied in the pioneering papers \cite{CS,CS1}, and the corresponding interior Schauder theory has reached substantially optimal results, see \cite{K2,S}. The boundary regularity in the fully nonlinear case also  parallels the linear one,  with some technical restrictions, see \cite{ROS1}. For more precise statements and wider bibliographic references, we refer to \cite{FRO2}.
\vskip2pt
\noindent
The picture for degenerate and singular fractional operators such as $(-\Delta)_p^s$ for $p\ne 2$ is less clear. On one hand, there are many available definitions of what may be considered a fractional version of the $p$-Laplacian. While the Gagliardo semi-norm
\[[u]_{s, p}:=\left(\iint_{\R^N\times\R^N}\frac{|u(x)-u(y)|^p}{|x-y|^{N+ps}}\,dx\,dy\right)^{\frac{1}{p}} \]
has a long history as an object of interest in the theory of Besov spaces, the operator $\fpl$, defined as the differential of $u\mapsto [u]_{s, p}^p/p$, has been considered for the first time in \cite{IN} as an approximation of the standard $p$-Laplacian for $s\to 1^-$. Other definitions of fractional $p$-Laplacians are proposed and studied in \cite{BCF, BCF1,CJ,L2} and are more related to the viscosity framework. In the variational framework (which is the one we adopt here) the most closely related operator is the so-called $H^{s, p}$-fractional Laplacian introduced in \cite{SS,SS1}.  
\vskip2pt
\noindent
Here, we are interested in the operator $(-\Delta)_p^s$ defined as above and the related equation
\beq\label{neq}
\fpl u = f(x) \quad \text{in $\Omega$.}
\eeq
For such an equation, interior H\"older regularity  has been established for the first time in \cite{DKP,DKP1} for the homogeneous case, and in \cite{KMS} for the non-homogeneous case. These results also cover variants of the operator having bounded measurable coefficients, and in this setting the most recent developments are achieved through a new class of fractional De Giorgi classes introduced in \cite{C} and exhibiting a purely nonlocal regularising effect. In these latter works the H\"older exponent obtained is unexplicitly small, but for the model case of the fractional $p$-Laplacian considered here a precise, and in many cases optimal, H\"older exponent can be derived. Indeed, in \cite{BLS} (see also \cite{DN}) it was proved that if $u$ solves \eqref{neq} with $f\in L^\infty(\Omega)$, then $u\in C^\alpha_{\rm loc}(\Omega)$ for any
\[0 < \alpha < \min\Big\{\frac{ps}{p-1},\,1\Big\}.\]
The same result has recently been extended to the singular case $p\in (1, 2)$ in \cite{GL}. {\em Global} regularity (i.e., up to the boundary) for the Dirichlet problem \eqref{dir}, still with a small H\"older exponent, is the subject of \cite{IMS,KKP}. Combining this regularity with the results of \cite{BLS,GL} ensures that if  $f\in L^\infty(\Omega)$, then $u\in C^s(\R^N)$ (which is optimal also for the fractional Laplacian), see Theorem \ref{opt}. It therefore makes sense to study the {\em fine}, or {\em weighted}, boundary regularity of solutions to \eqref{dir}, in the sense of \cite{ROS}. In \cite{IMS1} it has been proved that indeed, if $p\ge 2$ and $f\in L^\infty(\Omega)$, then $u/\ds$ admits a H\"older continuous extension to $\overline\Omega$ (and hence to $\R^N$), with an undetermined small H\"older exponent and a uniform estimate of the H\"older norm of such extension. Our aim in this paper is to prove an analogous result for the {\em singular} case $p\in(1,2)$, which is precisely Theorem \ref{main}.

\subsection{Motivations and applications}
The motivation for considering this type of weighted boundary regularity is the following. Given the possibly singular behavior of $u$ near $\partial\Omega$, H\"older continuity of $u/\ds$ is the natural fractional counterpart of global $C^{1,\alpha}$-regularity for the classical $p$-Laplace equation, obtained under very general conditions in \cite{L1}. The analogy is intuitive as soon as we consider the fractional order derivative at a point $x\in\partial\Omega$ along the inner normal direction $\nu$
\[\frac{\partial u}{\partial\nu^s}(x) = \lim_{t\to 0_+}\frac{u(x+t\nu)}{t^s} \sim \lim_{y\to x}\frac{u(y)}{\ds(y)}.\]
\vskip2pt
\noindent
The applications of Theorem \ref{main}, similar to those of the result of \cite{L1} in the local case, are mostly related to the following generalization of problem \eqref{dir}:
\beq\label{non}
\begin{cases}
\fpl u = f(x,u) & \text{in $\Omega$} \\
u = 0 & \text{in $\R^N\setminus\Omega$,}
\end{cases}
\eeq
where $f:\Omega\times\R\to\R$ is a Carath\'eodory mapping subject to a subcritical or critical growth condition (see \cite{ILPS} for a detailed functional-analytic framework for such problems). For all $\alpha\in[0,1]$ define the weighted H\"older space
\[C^\alpha_s(\overline\Omega) = \Big\{u\in C^0(\overline\Omega):\,\frac{u}{\ds} \ \text{has a $\alpha$-H\"older continuous extension to $\overline\Omega$}\Big\}.\]
Clearly, $C^\alpha_s(\overline\Omega)$ is compactly embedded into $C^0_s(\overline\Omega)$ for all $\alpha>0$. This, in conjunction with the uniform estimate of Theorem \ref{main} and the {\em a priori} bound of \cite{CMS}, gives rise for instance to the following interesting applications:
\begin{itemize}[leftmargin=1cm]
\item[$(a)$] {\em Sobolev vs.\ H\"older minima.} Using Theorem \ref{main}, it can be seen that the local minimizers of the energy functional corresponding to \eqref{non} in the Sobolev space $W^{s,p}_0(\Omega)$ and in $C^0_s(\overline\Omega)$, respectively, coincide. This is a valuable information in nonlinear analysis, when aiming at multiplicity results via variational methods (see \cite{IMS2} for the case $p\ge 2$, while the case $p\in(1,2)$ will be considered in a forthcoming paper).
\item[$(b)$] {\em Strong minimum/comparison principles}. In \cite{IMP} some general minimum and comparison results have been proved for sub-supersolutions of fractional $p$-Laplacian problems, for instance we recall the following Hopf-type lemma: under very general conditions on $f$, for any solution $u$ of problem \eqref{non} we have
\[\inf_\Omega\frac{u}{\ds} > 0.\]
By Theorem \ref{main}, the above information rephrases in the topological form $u\in {\rm int}(C^0_s(\overline\Omega)_+)$, which again can be used in several existence and multiplicity results.
\item[$(c)$] {\em Extremal solutions in an interval.} Let $\underline u\le\overline u$ be a sub-supersolution pair for \eqref{non}. Then, using Theorem \ref{main} it can be seen that the set of all solutions $u$ s.t.\ $\underline u\le u\le\overline u$ in $\Omega$ is nonempty, compact in both $W^{s,p}_0(\Omega)$ and in $C^0_s(\overline\Omega)$, and it admits a smallest and a largest element with respect to the pointwise ordering. Such structural properties have wide use in topological methods (see \cite{FI} for the case $p\ge 2$).
\end{itemize}

\subsection{Sketch of the proof}
For $p\ge 2$, Theorem \ref{main} is simply \cite[Theorem 1.1]{IMS1}. Thus, we will prove only the case $p\in(1,2)$, which requires a wholly different approach.
\vskip2pt
\noindent
The strategy of proof is initially based on the barrier techniques introduced in \cite{IMS1}. The main point is to construct lower and upper estimates in the form of weak Harnack inequalities for the function $u/\ds$ in terms of its {\em nonlocal excess} 
\[L(u,x_0,m,R)=\Big[\dashint_{\tilde{B}_{x_0,R}}\Big|\frac{u}{\ds}-m\Big|^{p-1}\,dx\Big]^\frac{1}{p-1}\]
where $x_0\in\partial\Omega$, $m\in\R$, $R>0$, and $\tilde{B}_{x_0, R}$ is a ball contained in $\Omega\cap B_{2R}(x_0)$ of radius comparable to $R$ and satisfying (see Figure \ref{fig1})
\[{\rm dist}\big(\tilde{B}_{x_0, R}, B_{R}(x_0)\big) \simeq R,\]
so that in particular $\tilde{B}_{x_0, R}$ is {\em disjoint} from $ B_{R}(x_0)$.
It turns out that the size of the excess of $u/\ds$, which measures its behavior {\em outside} the ball $B_R(x_0)$, provides quantitative estimates on its behavior {\em inside} $B_R(x_0)$ when coupled with a bound on $\fpl u$. This is possible due to the non-local nature of $\fpl$. 
\vskip2pt
\noindent
Precisely, we will prove the following. Let $D_R=B_R(x_0)\cap \Omega$, $K>0$, and $m\ge 0$, then
\beq\label{hil}
\begin{cases}
\fpl u \ge -K & \text{in $D_R$} \\
u \ge m\ds & \text{in $\R^N$}
\end{cases}
\quad \Longrightarrow \quad \inf_{D_{R/2}}\Big(\frac{u}{\ds}-m\Big) \ge \sigma L(u,x_0,m,R)-C(m,K,R) 
\end{equation}
with a fixed $\sigma>0$ only depending on the data $N$, $p$, $s$, and $\Omega$. Similarly, for any $M\ge 0$  
\beq\label{hiu}
\begin{cases}
\fpl u \le K & \text{in $D_R$} \\
u \le M\ds & \text{in $\R^N$}
\end{cases}
\quad \Longrightarrow \quad \inf_{D_{R/2}}\Big(M-\frac{u}{\ds}\Big) \ge \sigma L(u,x_0,M,R)-C(M,K,R). 
\end{equation}
The assumption of a {\em global} point-wise control of $u$ by multiples of $\ds$ is needed to apply comparison principles for the nonlocal operator $\fpl$ and represents the main difference from the local case, as well as the source of many new difficulties which will be detailed below. In order to prove these Harnack inequalities we will have to distinguish the cases when the excess is comparatively large or small, according to the size of the ratio $L(u,x_0,m,R)/m$ (resp.\ $L(u,x_0,M,R)/M$). It is worth mentioning that, differently from what happens in the local case, the proof of \eqref{hiu} is considerably more involved than the one of \eqref{hil}, essentially because the condition $u\le M\ds$ gives no sign information on $u$ near $x_0$, which can then be very small in absolute value in relatively large subsets of $B_R$. Since the operator $\fpl$ is singular precisely when $u\simeq 0$, it is then more delicate to infer bounds on $u$ from bounds on $\fpl u$, compared to the degenerate case $p\ge 2$. 
\vskip2pt
\noindent
The peculiar form of the constants $C(m,K,R)$, $C(M,K,R)$ appearing in \eqref{hil}, \eqref{hiu} respectively, is discussed in Example \ref{cmr} below and plays a major role. We just note here that we must aim at its optimal form, in terms of asymptotic behavior with respect to its arguments. The reason is the following. In order to infer from \eqref{hil}, \eqref{hiu} the desired H\"older regularity result we adapt Krylov's method (see \cite{K3}), applying these inequalities at the scales $R_n=R_0/2^n$ to deduce a decay in oscillation for $u/\ds$ on the corresponding sets $D_{R_n}$. Indeed, from \eqref{hil}, \eqref{hiu} one readily derives for the solutions of
\beq\label{bil}
\begin{cases}
|\fpl u| \le K & \text{in $D_R$} \\
m\ds \le u \le M\ds & \text{in $\R^N$}
\end{cases}
\eeq
the following estimate, holding for suitable $\theta\in(0,1)$:
\beq\label{hib}
\underset{D_{R/2}}{\rm osc}\,\frac{u}{\ds} \le \theta\,\underset{D_R}{\rm osc}\,\frac{u}{\ds}+C(m,R,K)+C(M,R,K).
\eeq
On one hand, \eqref{hib} looks promising: if one can prove good controls on the last two terms as $R\to 0$, the claimed decay in oscillation will follow.  On the other hand, an iterative argument ensures that  for suitable $m_n, M_n$ it holds $m_n\ds\le u\le M_n\ds$, but only in $D_{R_n}$, thus prejudicing the global bound in \eqref{bil}. Therefore we have to apply the weak Harnack inequalities to the truncated function
\[\tilde{u}_n=\max\big\{\min\{u,M_n\ds\},\,m_n\ds\big\},\]
with $m_n, M_n$ iteratively determined at scale $R_n$, which satisfies the bilateral bound in the whole $\R^N$. Due to the nonlocal nature of $\fpl$,  however, these truncations worsen the bound $|\fpl u|\le K$, so that $\tilde u_n$ satisfies
\[\begin{cases}
|\fpl\tilde u_n| \le \tilde K_n & \text{in $D_R$} \\
m_n\ds \le \tilde u_n \le M_n\ds & \text{in $\R^N$,}
\end{cases}\]
for a possibly much bigger $\tilde{K}_n$. Therefore \eqref{hib} holds true with constants depending on $\tilde{K}_n$ (which is rather implicitly constructed by induction), and these have to be iteratively estimated. This purely non-local phenomenon and the corresponding issues have been faced and overcome for the first time in \cite{ROS}, dealing with the linear case $p=2$, through a strong induction argument taking advantage of the simple form of the constants $C(m,R,K)$, $C(M,R,K)$ appearing in the corresponding weak Harnack inequalities. In the nonlinear case, the asymptotic behavior of these constants with respect to their arguments is rather more involved, but still the case $p>2$ has been dealt with in \cite{IMS1}. The singular case $p\in(1,2)$ considered here turns out to be even more delicate and requires a different argument based on the co-area formula.

\subsection{Plan of the paper}
The structure of the paper is the following: in Section \ref{sec2} we collect some useful results and definitions; in Section \ref{sec3} we prove a lower bound for supersolutions of fractional $p$-Laplacian equations; in Section \ref{sec4} we prove an upper bound for subsolutions; in Section \ref{sec5} we prove an oscillation bound for functions with a bounded fractional $p$-Laplacian, and finally obtain weighted H\"older regularity; Appendix \ref{appa} is devoted to the proof of some elementary inequalities, and Appendix \ref{appb} to the proof of a barrier proposition (integrating a similar result of \cite{IMS1}).
\vskip4pt
\noindent
{\bf Notations.} Throughout the paper, for any $U\subset\R^N$ we shall set $U^c=\R^N\setminus U$ and $\chi_U$ denotes the characteristic function of $U$. If $U$ is measurable,  $|U|$ stands for its $N$-dimensional Lebesgue measure. Open balls in $\R^N$ with center $x$ and radius $R$ will be denoted by $B_R(x)$, omitting the $x$-dependence if $x=0$. We denote by ${\rm dist}(x,U)$ the infimum of $|x-y|$ as $y\in U$, and we set $d_U(x)={\rm dist}(x,U^c)$. For any two measurable functions $u,v:U\to\R$, $u\le v$ in $U$ will mean that $u(x)\le v(x)$ for a.e.\ $x\in U$ (and similar expressions). The positive (resp., negative) part of $u$ is denoted $u_+$ (resp., $u_-$), while $u\land v=\min\{u,v\}$ and $u\lor v=\max\{u,v\}$. For brevity, we will set  for all $x\in\R$, $q>0$
\[x^q = \begin{cases}
|x|^{q-1}x&\text{if $x\ne 0$}\\
0&\text{if $x=0$}.
\end{cases}\]
Moreover, $C$ will denote a positive constant whose value may change case by case and whose dependance on the parameters will be specified each time.

\section{Preliminaries}\label{sec2}

\begin{figure}
\centering
\begin{tikzpicture}[scale=1.3]
\filldraw (0, 0) circle (1pt);
\draw (0, 0) node[below]{$x$};
\draw[clip] (-4,3.5) to [out=-20, in=180] (0, 0) to [out=0, in=200] (3.5,2) to [out=20, in=-100] (5, 3.5);
\draw[thin] (0,0) circle (2.25);
\draw (3.5,3) node{$\Omega$};
\draw (0, 2.625) node{$\tilde{B}_{x, R}$};
\draw (0, 2.625) circle (0.36);
\draw (0, 0.9) node{$D_{R}$};
\draw[clip] (0, 0) circle (3);
\draw[dashed] (0, 0) -- (4,4);
\draw (1.32,1.32) node[fill=white]{$2R$};
\draw[thick] (0,0) circle (1.5);
\draw[clip] (0,0) circle (1.5);
\draw[dashed] (0, 0) -- (-2, 2);
\draw (-0.6, 0.6) node[fill=white]{$R$};
\end{tikzpicture}
\caption{The ball $\tilde{B}_{x, R}$, with center on the normal direction.}
\label{fig1}
\end{figure}
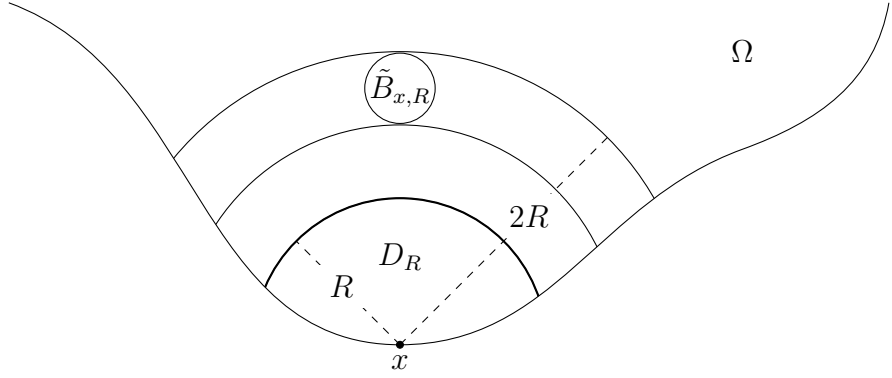

\noindent
In this section we recall some notions and results that will be used in our argument. 

\subsection{Properties of the distance function}
We begin with some geometrical remarks, referring to \cite{IMS1} for details. Since $\Omega$ is $C^{1,1}$-smooth, it satisfies the interior sphere property with optimal (half) radius
\[\rho_\Omega = \sup\Big\{R>0:\,\text{for all $x\in\partial\Omega$ there is $y\in\Omega$ s.t.\ $B_{2R}(y)\subseteq\Omega$, $x\in\partial B_{2R}(y)$}\Big\} > 0.\]
The distance function ${\rm d}_\Omega$ fulfils $|\nabla {\rm d}_\Omega|=1$ a.e.\, and it is globally $C^{1,1}$ on the closure of $\{x\in \Omega: 0<{\rm d}_\Omega(x)<\rho_\Omega\}$. Moreover, the nearest point projection
\[\Pi(x)={\rm Argmin}\{|x-y|:y\in \partial\Omega\}\]
is well defined and uniformly Lipschitz on $\{x\in \Omega: 0<{\rm d}_\Omega<\rho_\Omega\}$ (see \cite{LS}). For all $x_0\in\partial\Omega$, $R\in(0,\rho_\Omega)$ we denote
\[D_R(x_0) = B_R(x_0)\cap\Omega,\]
omitting the dependance on $x_0$ when $x_0=0$.
In addition, there exists a ball $\tilde B_{x_0,R}\subset\Omega$ (see Figure \ref{fig1}), with radius $R/4$, s.t.\ $\tilde B_{x_0,R} \subset D_{2R}(x_0)\setminus D_{3R/2}(x_0)$ and
\beq\label{geo}
\inf_{x\in\tilde B_{x_0,R}}{\rm d}_\Omega(x) \ge \frac{3R}{2}.
\eeq
Clearly, the boundary of $D_R$ fails to be smooth in general, hence the interior sphere property does not hold. So, in our following results, we will need to use the regularized set $A_R(x_0)$, defined as in \cite[Lemma 3.1]{IMS1} by
\beq\label{ard}
A_R(x_0) = \bigcup\Big\{B_r(y):\,r\ge\frac{R}{8},\,B_r(y)\subset D_R(x_0)\Big\}
\eeq
(see Figure \ref{fig2}). By construction, $A_R(x_0)$ satisfies the interior sphere property with $\rho_{A_R(x_0)}\ge R/16$. Also, this set enjoys some useful properties:

\begin{lemma}\label{arp}
Let $x_0\in\partial\Omega$, $R\in(0,\rho_\Omega)$, $A_R(x_0)\subseteq\Omega$ be defined as in \eqref{ard}. Then
\begin{enumroman}
\item\label{arp1} $D_{3R/4}(x_0)\subset A_R(x_0)\subset D_R(x_0)$;
\item\label{arp2} for all $x\in D_{3R/4}(x_0)$
\[\frac{{\rm d}_\Omega(x)}{6}\le{\rm d}_{A_R(x_0)}(x)\le{\rm d}_\Omega(x).\]
\end{enumroman}
\end{lemma}
\begin{proof}
For simplicity, let $x_0=0\in\partial\Omega$ and omit the center in all notations. Note that $A_R\subseteq D_R$ by construction and ${\rm d}_{A_R}\le {\rm d}_\Omega$ trivially from $A_R\subseteq \Omega$. Fix $x\in D_{3R/4}$, and distinguish two cases:
\begin{itemize}[leftmargin=1cm]
\item[$(a)$] If ${\rm d}_\Omega(x)>R/8$, then $B_{R/8}(x)\subseteq\Omega$, and for all $z\in B_{R/8}(x)$ we have
\[|z| \le |z-x|+|x| \le \frac{R}{8}+\frac{3R}{4} < R.\]
So, $B_{R/8}(x)\subseteq D_R$, hence $B_{R/8}(x)\subseteq A_R$ by \eqref{ard}, in particular $x\in A_R$, proving \ref{arp1}. This in turn implies
\[{\rm d}_{A_R}(x)\ge  \frac{R}{8}\]
while in $D_{3R/4}$ it holds ${\rm d}_\Omega(y)\le 3R/4$. Chaining these inequalities proves the first inequality in \ref{arp2}.
\item[$(b)$] If ${\rm d}_\Omega(x)\le R/8$, then let $\bar x\in\partial\Omega$ be one point s.t.\
\[{\rm d}_\Omega(x) = |x-\bar x| = r,\]
and $\bar y\in\Omega$ be s.t.\ $B_{R/8}(\bar y)$ is tangent to $\partial\Omega$ at $\bar x$. Since $B_r(x)$ is tangent to $\partial\Omega$ at $\bar x$ as well and $r<R/8$, we infer $B_r(x)\subset B_{R/8}(\bar y)$ and $|x-\bar y|\le R/8$. For all $z\in B_{R/8}(\bar y)$ we have
\[|z| \le |z-\bar y|+|\bar y-x|+|x| < \frac{R}{8}+\frac{R}{8}+\frac{3R}{4} = R.\]
So $x\in B_{R/8}(\bar y)\subseteq D_R$, hence $x\in A_R$, which proves \ref{arp1}. Also, from $B_r(x)\subseteq A_R$ we get
\[{\rm d}_{A_R}(x)\ge r={\rm d}_\Omega(x),\]
proving the first inequality in \ref{arp2}.
\end{itemize}
In both cases we conclude.
\end{proof}

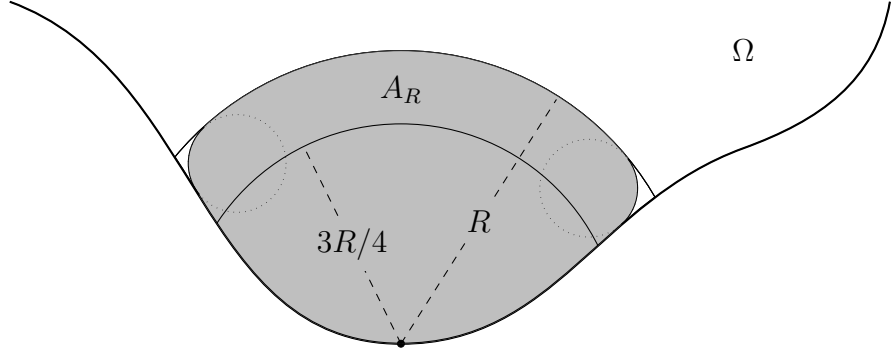
\begin{figure}
\centering
\begin{tikzpicture}[scale=1.3]
\filldraw[lightgray] (-1.675,1.845) circle (0.5);
\filldraw[lightgray] (1.92,1.592) circle (0.5);
\filldraw (0, 0) circle (1pt);
\draw (3.5,3) node{$\Omega$};
\draw[very thin] (-1.675,1.845) circle (0.5);
\draw[very thin] (1.92,1.592) circle (0.5);
\draw[thick] (-4,3.5) to [out=-20, in=180] (0, 0) to [out=0, in=200] (3.5,2) to [out=20, in=-100] (5, 3.5);
\draw[clip] (-4,3.5) to [out=-20, in=180] (0, 0) to [out=0, in=200] (3.5,2) to [out=20, in=-100] (5, 3.5);
\draw (0, 0) circle (3);
\filldraw[lightgray] (0, 0) circle (2.56);
\draw[clip] (0, 0) circle (3);
\filldraw[lightgray]  (0, 0) -- (2.5, 2.05) -- (0, 4) -- (-2.26, 2.5) -- (0, 0);
\draw[very thin, dotted] (1.92,1.592) circle (0.5);
\draw[very thin, dotted] (-1.675,1.845) circle (0.5);
\draw (0, 0) circle (2.25);
\filldraw (0, 0) circle (1pt);
\draw (0, 2.6) node{$A_{R}$};
\draw[dashed] (0,0) -- node[ fill=lightgray]{$R$} (1.59, 2.5);
\draw[dashed] (0,0) -- node[fill=lightgray]{$3R/4$} (-1,2.03);
\draw[thick] (-4,3.5) to [out=-20, in=180] (0, 0) to [out=0, in=200] (3.5,2) to [out=20, in=-100] (5, 3.5);
\draw (0, 0) circle (3);
\end{tikzpicture}
\caption{The regularized set $A_{R}$ in gray satisfies $D_{3R/4}\subset A_R\subset D_R$.}
\label{fig2}
\end{figure}

\noindent
At some step of our proof we will need to estimate the $(N-1)$-dimensional Hausdorff measure (denoted $\mathcal{H}^{N-1}$) of the level set
\[S_{R,\xi}(x_0) = \big\{x\in D_R(x_0):\,{\rm d}_\Omega(x)=\xi\big\},\]
for some $x_0\in\partial\Omega$, $R\in(0,\rho_\Omega)$, and $\xi\in(0,R)$. We have the following result:

\begin{lemma}\label{lev}
Let $x_0\in\partial\Omega$, $R\in(0,\rho_\Omega)$, and $\xi>0$. Then, there exists $C=C(\Omega)>0$ s.t.\
\[\mathcal{H}^{N-1}(S_{R,\xi}(x_0)) \le CR^{N-1}.\]
\end{lemma}
\begin{proof}
Since ${\rm d}_\Omega(x)\le\rho_\Omega$ for all $x\in D_R(x_0)$, we may assume $\xi\le\rho_\Omega$. By the implicit function theorem $S_{R,\xi}(x_0)$ is a Lipschitz $(N-1)$-dimensional submanifold of $\R^N$ and the metric projection $\Pi_\Omega:D_R(x_0)\to\partial\Omega$ has a uniform Lipschitz bound. By the area formula
\[\mathcal{H}^{N-1}(S_{R,\xi}(x_0)) \le C\mathcal{H}^{N-1}(\Pi_\Omega(S_{R, \xi}(x_0)))\le C\mathcal{H}^{N-1}(\Pi_\Omega(D_{R}(x_0))),\]
for some $C>0$ depending on $\Omega$. Also, by the Lipschitz continuity of $\Pi_\Omega$, we can find $\eta\ge 1$ depending on $\Omega$ s.t.\
\[\Pi_\Omega(D_R(x_0)) \subseteq B_{\eta R}(x_0)\cap\partial\Omega.\]
Therefore, by the regularity of $\partial\Omega$ we infer
\[\mathcal{H}^{N-1}(S_{R,\xi}(x_0)) \le C\mathcal{H}^{N-1}(B_{\eta R}(x_0)\cap\partial\Omega) \le CR^{N-1},\]
with $C>0$ depending on $\Omega$.
\end{proof}

\subsection{Functional setting}

We will now introduce the functional spaces that we are going to work with, referring to \cite{L} for details. Fix $p>1$, $s\in(0,1)$, $U\subseteq\R^N$ and for all measurable $u:U\to\R$ define the Gagliardo seminorm
\[[u]_{s,p,U} = \Big[\iint_{U\times U}\frac{|u(x)-u(y)|^p}{|x-y|^{N+ps}}\,dx\,dy\Big]^\frac{1}{p}.\]
The basic fractional Sobolev space is defined by
\[W^{s,p}(U) = \big\{u\in L^p(U):\,[u]_{s,p,U}<\infty\big\},\]
while we set
\[W^{s,p}_0(U) = \big\{u\in W^{s,p}(\R^N):\,u=0 \ \text{in $U^c$}\big\}.\]
If $U$ has finite measure, the latter is a uniformly convex, separable Banach space under the norm $\|u\|=[u]_{s,p,U}$, with dual space $W^{-s,p'}(U)=(W^{s,p}_0(U))^*$. Also, if $U$ is bounded we define
\[\widetilde{W}^{s,p}(U) = \Big\{u\in L^p_{\rm loc}(\R^N):\,u\in W^{s,p}(V) \ \text{for some $V\Supset U$,} \ \int_{\R^N}\frac{|u(x)|^{p-1}}{(1+|x|)^{N+ps}}\,dx < \infty\Big\}.\]
Such space is the natural framework for the study of the fractional $p$-Laplacian. Indeed, by \cite[Lemma 2.3]{IMS} we can define a continuous, monotone operator $\fpl:\widetilde{W}^{s,p}(U)\to W^{-s,p'}(U)$ by setting for all $u\in\widetilde{W}^{s,p}(U)$, $\varphi\in W^{s,p}_0(U)$
\[\langle\fpl u,\varphi\rangle = \iint_{\R^N\times\R^N}\frac{(u(x)-u(y))^{p-1}(\varphi(x)-\varphi(y))}{|x-y|^{N+ps}}\,dx\,dy.\]
Such definition agrees with the one given in Section \ref{sec1}, since the restriction of $\fpl$ to $\w$ coincides with the gradient of the functional $u\mapsto [u]_{s,p,\R^N}/p$. We note that, at least for $p\ge 2$ and $u$ smooth enough, the fractional $p$-Laplacian allows for the following formulation:
\[\fpl u(x) = 2\lim_{\eps\to 0^+}\int_{B_\eps^c(x)}\frac{(u(x)-u(y))^{p-1}}{|x-y|^{N+ps}}\,dy.\]
Let us focus on the equation
\beq\label{feq}
\fpl u = f(x) \quad \text{in $U$}
\eeq
(without Dirichlet conditions), with $f\in L^\infty(U)$ and $u\in\widetilde W^{s,p}(U)$. We say that $u$ is a {\em (weak) supersolution} of \eqref{feq} if for all $\varphi\in W^{s,p}_0(U)_+$
\[\langle\fpl u,\varphi\rangle \ge \int_U f(x)\varphi(x)\,dx.\]
The definition of a {\em subsolution} is analogous. Finally, we say that $u$ is a {\em (weak) solution} of \eqref{feq} if it is both a super- and a subsolution. Accordingly, a solution of the Dirichlet problem \eqref{dir} is a function $u\in W^{s,p}_0(\Omega)$ s.t.\ for all $\varphi\in W^{s,p}_0(\Omega)$
\[\langle\fpl u,\varphi\rangle = \int_\Omega f(x)\varphi(x)\,dx.\]
All similar expressions throughout the paper will be meant in such weak sense.
\vskip2pt
\noindent
For the reader's convenience, we will finally recall some useful properties of (super-, sub-) solutions. Such properties are proved in \cite{IMS1}, but we remark that they hold for {\em any} $p>1$. We begin with a weak comparison principle:

\begin{proposition}\label{wcp}
{\rm\cite[Proposition 2.1]{IMS1}} Let $u,v\in\widetilde{W}^{s,p}(U)$ satisfy
\[\begin{cases}
\fpl u \le \fpl v & \text{in $U$} \\
u \le v & \text{in $U^c$.}
\end{cases}\]
Then, $u\le v$ in $\R^N$.
\end{proposition}

\noindent
The next result is a nonlocal superposition principle:

\begin{proposition}\label{spp}
{\rm\cite[Proposition 2.6]{IMS1}} Let $U\subset\R^N$ be bounded, $u\in\widetilde{W}^{s,p}(U)$, $v\in L^1_{\rm loc}(\R^N)$, $V={\rm supp}(v-u)$ satisfy $U\Subset V^c$ and
\[\int_V\frac{|v(y)|^{p-1}}{(1+|y|)^{N+ps}}\,dy < \infty.\]
Set for all $x\in\R^N$
\[w(x) = \begin{cases}
u(x) & \text{if $x\in V^c$} \\
v(x) & \text{if $x\in V$.}
\end{cases}\]
Then $w\in\widetilde{W}^{s,p}(U)$ and  
\[\fpl w(x) = \fpl u(x)+2\int_V\frac{(u(x)-v(y))^{p-1}-(u(x)-u(y))^{p-1}}{|x-y|^{N+ps}}\,dy\]
weakly in $U$.
\end{proposition}

\noindent
Finally, for all measurable function $u$, $x\in\R^N$, $m\in\R$, and $R>0$, we define the nonlocal excess as in Section \ref{sec1}:
\beq\label{lum}
L(u,x,m,R) = \Big[\dashint_{\tilde B_{x,R}}\Big|\frac{u(y)}{\ds(y)}-m\Big|^{p-1}\,dy\Big]^\frac{1}{p-1}.
\eeq
The quantity $L(u,x,m,R)$ will play a crucial role in the subsequent arguments. Also, for all $q>0$, $s\in(0,1)$ we borrow from \cite{DKP} the (slightly modified) definition of the nonlocal tail
\beq\label{tail}
{\rm tail}_q(u,x,R) = \Big[\int_{\Omega\cap B_R^c(x)}\frac{|u(y)|^q}{|x-y|^{N+s}}\,dy\Big]^\frac{1}{q}.
\eeq
As usual, we omit $x$ whenever $x=0$.

\subsection{Optimal regularity up to the boundary}

As pointed out in Section \ref{sec1}, combining interior H\"older estimates from \cite{BLS,GL} and boundary estimates from \cite{IMS1} we can obtain an optimal global regularity result for solutions of \eqref{dir}. Here we prove this assertion.
\vskip2pt
\noindent
First we recall a special case of the more general results obtained for the degenerate and singular cases, respectively, in \cite{BLS,GL}, using the following definition of nonlocal tail for any $u$, $x_1\in\R^N$, $R>0$:
\[{\rm Tail}(u,x_1,R) = \Big[R^{ps}\int_{B_R^c(x_1)}\frac{|u(x)|^{p-1}}{|x-x_1|^{N+ps}}\,dx\Big]^\frac{1}{p-1}.\]

\begin{proposition}\label{lin}
Let $U\subset\R^N$ be open and bounded, $u\in\widetilde{W}^{s,p}(U)$ be a local weak solution of \eqref{feq}, with $f\in L^\infty(U)$, and let $\gamma$ satisfy
\[0 < \gamma < \min\Big\{1,\,\frac{ps}{p-1}\Big\}.\]
Then, $u\in C^\gamma_{\rm loc}(U)$ and there exists $C=C(N,p,s,\gamma)>0$ s.t.\ for all $x_1\in U$, $R>0$ s.t.\ $B_{4R}(x_1)\subseteq U$
\beq\label{lin1}
[u]_{C^\gamma(B_{R/8}(x_1))} \le \frac{C}{R^\gamma}\Big[\|u\|_{L^\infty(B_R(x_1))}+R^\frac{ps}{p-1}\|f\|_{L^\infty(U)}^{\frac{1}{p-1}}+{\rm Tail}(u,x_1,R)\Big].
\eeq
\end{proposition}

\noindent
We also recall a technical lemma, contained in the proof of \cite[Theorem 1.2]{ROS}, which will also be used to prove our main result:

\begin{lemma}\label{ros}
Let $\Omega\subset\R^N$ be a bounded domain with $C^{1,1}$-smooth boundary, $v\in L^\infty(\Omega)$, $\gamma\in(0,1)$, $M>0$, $\nu\ge 0$ satisfy

\begin{enumroman}
\item\label{ros1} $\|v\|_{L^\infty(\Omega)}\le M$;
\item\label{ros2} for all $x_1\in\Omega$ s.t.\ ${\rm d}_\Omega(\bar x)=4R$, $v\in C^\gamma(B_{R/8}(x_1))$ with
\[[v]_{C^\gamma(B_{R/8}(x_1))} \le M(1+R^{-\nu});\]
\item\label{ros3} for all $x_0\in\partial\Omega$, $r>0$ small enough
\[\underset{D_r(x_0)}{\rm osc}\,v \le Mr^\gamma.\]
\end{enumroman}
Then $v\in C^\alpha(\overline\Omega)$ with $\alpha=\gamma^2/(\gamma+\nu)\in(0,1)$ and there exists $C=C(M, \gamma, \nu)>0$ s.t.\
\[[v]_{C^{\bar\alpha}(\overline\Omega)}\le C.\]
\end{lemma}

\noindent
Here follows the optimal regularity result:

\begin{theorem}\label{opt}
Let $\Omega\subset\R^N$ be a bounded domain with $C^{1,1}$-smooth boundary, $f\in L^\infty(\Omega)$, $u\in\w$ be a weak solution of \eqref{dir}. Then, $u\in C^s(\R^N)$ and there exists $C=C(N,p,s,\Omega)>0$ s.t.\
\[\|u\|_{C^s(\R^N)} \le C\|f\|_{L^\infty(\Omega)}^{\frac{1}{p-1}}.\]
\end{theorem}
\begin{proof}
Assume $u\neq 0$, and by $(p-1)$-homogeneity of $\fpl$ we may as well assume $\|f\|_{L^\infty(\Omega)}=1$. By \cite[Theorem 4.4]{IMS} there exists $C=C(N,p,s,\Omega)>0$ s.t.\ for all $u\in\R^N$
\beq\label{opt1}
|u(x)| \le C\ds(x).
\eeq
We aim at applying Lemma \ref{ros} to $u$, with $\gamma=s$, $\nu=0$ and a convenient $M>0$ depending on $N$, $p$, $s$, and $\Omega$. First, from \eqref{opt1} and boundedness of $\Omega$ we have
\[\|u\|_{L^\infty(\Omega)} \le C{\rm diam}(\Omega)^s,\]
hence $u$ satisfies hypothesis \ref{ros1} of Lemma \ref{ros} with $M=C{\rm diam}(\Omega)^s$. Also, for all $x_0\in\partial\Omega$ and $r\in(0,\rho_\Omega)$ we have by \eqref{opt1}
\[\underset{D_r(x_0)}{\rm osc}\,u \le 2C\sup_{D_r(x_0)}\,u \le 2Cr^s,\]
hence $u$ satisfies \ref{ros3} as well, with a possibly bigger $M$. In order to check hypothesis \ref{ros2}, we fix $x_1\in\Omega$ s.t.\ ${\rm d}_\Omega(x_1)=4R$, set $\gamma=s$ again, and invoke Proposition \ref{lin} (with $U=\Omega$), getting $u\in C^\gamma(B_{R/8}(x_1))$. Besides, from \eqref{lin1}, \eqref{opt1}, and $\|f\|_{L^\infty(\Omega)}=1$ we have
\beq\label{opt2}
[u]_{C^\gamma(B_{R/8}(x_1))} \le \frac{C}{R^s}\big[R^s+R^\frac{ps}{p-1}+{\rm Tail}(u,x_1,R)\big],
\eeq
with $C=C(N,p,s,\Omega)>0$. The second term in the right hand side is estimated as
\[R^\frac{ps}{p-1} \le {\rm diam}(\Omega)^\frac{s}{p-1}R^s.\]
For the tail term, let $\bar x\in\partial\Omega$ be one point minimizing the distance from $x_1$, and for all $x\in\R^N$ we have
\[{\rm d}_\Omega(x) \le |x-\bar x| \le |x-x_1|+4R,\]
which by subadditivity of $t\mapsto t^s$ in $[0,\infty)$ and \eqref{opt1} again implies
\[|u(x)| \le C\big(|x-x_1|^s+R^s).\]
Using the above estimate (and different subadditivity properties depending on $p\ge 2$ or $1<p<2$, respectively) we have
\begin{align*}
\int_{B_r^c(x_1)}\frac{|u(x)|^{p-1}}{|x-x_1|^{N+ps}}\,dx &\le C\int_{B_R^c(x_1)}\frac{|x-x_1|^{(p-1)s}+R^{(p-1)s}}{|x-x_1|^{N+ps}}\,dx \\
&\le C\int_{B_R^c(x_1)}\frac{dx}{|x-x_1|^{N+s}}+CR^{(p-1)s}\int_{B_R^c(x_1)}\frac{dx}{|x-x_1|^{N+ps}} \le \frac{C}{R^s},
\end{align*}
still with $C>0$ depending on $N$, $p$, $s$, and $\Omega$. So
\[{\rm Tail}(u,x_1,R) \le CR^s.\]
Plugging these estimates into \eqref{opt2} we get
\[[u]_{C^s(B_{R/8}(x_1))} \le C(N,p,s,\Omega).\]
So hypothesis \ref{ros2} is satisfied (with $\nu=0$). By Lemma \ref{ros} we get $u\in C^s(\overline\Omega)$ and $[u]_{C^s(\overline\Omega)}\le C$, with $C=C(N,p,s,\Omega)>0$. Recalling \eqref{opt1} once again and $u=0$ in $\Omega^c$, we conclude.
\end{proof}

\subsection{Torsion functions and barriers}

An important auxiliary problem, in the study of nonlocal regularity, is the following Dirichlet problem for the torsion equation:
\beq\label{tor}
\begin{cases}
\fpl v = 1 & \text{in $U$} \\
v = 0 & \text{in $U^c$.}
\end{cases}
\eeq
According to the shape of $U$, the (unique) solution of \eqref{tor} enjoys useful properties, including a Hopf type lemma and a global subsolution property:

\begin{proposition}\label{top}
Let $U\subset\R^N$ be bounded and satisfy the interior sphere property with radius $\rho_U>0$, $v\in W^{s,p}_0(U)\cap C^0(U)$ be the solution of \eqref{tor}. Then:
\begin{enumroman}
\item\label{top1} there exists $C=C(N,p,s)>1$ s.t.\ for all $x\in\R^N$
\[v(x) \ge \frac{1}{C}\rho_U^\frac{s}{p-1}{\rm d}_U^s(x);\]
\item\label{top2} $v$ satisfies weakly in $\R^N$
\[\fpl v \le 1;\]
\item\label{top3} there exists $C=C(N,p,s)>0$ s.t.\ for all $x\in U$
\[v(x) \le C{\rm diam}(U)^\frac{ps}{p-1}.\]
\end{enumroman}
\end{proposition}
\begin{proof}
Properties \ref{top1}, \ref{top2} are proved exactly as in \cite[Lemmas 2.3, 2.4]{IMS1}. Regarding \ref{top3}, let ${\rm diam}(U)=2R>0$ and find $x_0\in\R^N$ s.t.\ $U\subseteq B_R(x_0)$. Further, let $u_R\in W^{s,p}_0(B_R(x_0))$ be the solution of the torsion problem
\[\begin{cases}
\fpl u_R = 1 & \text{in $B_R(x_0)$} \\
u_R = 0 & \text{in $B_R^c(x_0)$.}
\end{cases}\]
Arguing as in \cite[Lemma 2.2]{IMS1} we find $C=C(N,p,s)>1$ s.t.\ for all $x\in\R^N$
\[\frac{1}{C}R^\frac{s}{p-1}{\rm d}_{B_R(x_0)}^s(x) \le u_R(x) \le CR^\frac{s}{p-1}{\rm d}_{B_R(x_0)}^s(x),\]
By \ref{top2} we have
\[\begin{cases}
\fpl v \le 1 = \fpl u_R & \text{in $B_R(x_0)$} \\
v = 0 \le u_R & \text{in $B_R^c(x_0)$.}
\end{cases}\]
By Proposition \ref{wcp}, we have $v\le u_R$ in $\R^N$. In particular, for all $x\in U$ we have
\[v(x) \le u_R(x) \le CR^\frac{s}{p-1}{\rm d}_{B_R(x_0)}^s(x) \le CR^\frac{ps}{p-1},\]
which implies \ref{top3}.
\end{proof}

\noindent
Finally, we recall some technical results which play a crucial role in the construction of barriers. The first comes from \cite{IMS1}:

\begin{proposition}\label{ptz}
{\rm\cite[Lemma 4.1]{IMS1}} Let $U\subset\R^N$ have a $C^{1,1}$-smooth boundary, $0\in\partial U$, $R\in(0,\rho_U/4)$, and $x_0\in B_{R/2}\cap U$. Then, there exist $v\in W^{s,p}_0(U)\cap C^0(\R^N)$, $C=C(N,p,s,U)>1$ s.t.\
\begin{enumroman}
\item $|\fpl v|\le CR^{-s}$ in $B_{2R}\cap U$;
\item $|v|\le CR^s$ in $B_{2R}\cap U$;
\item $v\ge C^{-1}{\rm d}_U^s$ in $(B_R\cap U^c)$;
\item $v(x_0)=0$.
\end{enumroman}
\end{proposition}

\noindent
The second is a slightly modified version of \cite[Lemma 3.4]{IMS1}, whose proof is postponed to Appendix \ref{appb} as it is only loosely related to the main subject of the present work:

\begin{proposition}\label{bar}
Let $U\subset\R^N$ have a $C^{1,1}$-smooth boundary, $0\in\partial U$, $\varphi\in C^\infty_c(B_1)$ be s.t.\ $0\le\varphi(x)\le 1$ for all $x\in B_1$, and for all $\lambda\in\R$, $R>0$, and $x\in\R^N$ set
\[v_\lambda(x) = \Big(1+\lambda\varphi\Big(\frac{x}{R}\Big)\Big){\rm d}_U^s(x).\]
Then, there exist $\rho'_U, \lambda_0, C>0$ depending on $N$, $p$, $s$, $U$ and $\varphi$, s.t.\ for all  $R\le\rho'_U$,  $|\lambda|\le\lambda_0$  
\[|\fpl v_\lambda| \le C\Big(1+\frac{|\lambda|}{R^s}\Big)\]
weakly in $B_{\rho'_U}\cap U$.
\end{proposition}

\section{Lower bound}\label{sec3}

\noindent
In this section we prove a lower bound for supersolutions of \eqref{dir}-type problems in domains of the type $D_R$, globally bounded from below by a positive multiple of the function $\ds$, which corresponds to the weak Harnack inequality \eqref{hil} seen in Section \ref{sec1}. The peculiar form of the involved constant $C(m,K,R)$ can be inferred via the following representative example, in which we assume $K\sim 1$ for simplicity:

\begin{example}\label{cmr}
Let $\Omega=B_1$, $x_0\in\partial\Omega$, $0<m<\mu$, and set for all $x\in\R^N$
\[u(x) = \begin{cases}
m\ds (x) & \text{if $x\in\tilde B_{x_0,R}^c$} \\
\mu\ds(x) & \text{if $x\in\tilde B_{x_0,R}$.}
\end{cases}\]
Then, the left hand side of \eqref{hil} vanishes, so that for any choice ensuring  $\fpl u\gtrsim -1$ in $D_R(x_0)$,  the optimal constant $C(m,1,R)$ in \eqref{hil} must fulfill 
\[C(m,1,R) \simeq L(u,x_0,m,R).\]
For $R>0$ small enough, an explicit computation yields
\[L(u,x_0,m,R)\simeq \mu-m, \qquad \inf_{D_R(x_0)}\fpl u\gtrsim m^{p-1}-\frac{\mu-m}{\mu^{2-p}R^s}.\]
We can choose $\mu>m$ according to the following cases:
\begin{itemize}[leftmargin=1cm]
\item[$(a)$] If $m\gtrsim 1$, then we choose $\mu-m\simeq mR^s$, so that for small $R$ we have in $D_R(x_0)$
\[\fpl u \gtrsim m^{p-1}-m\mu^{p-2} \gtrsim -1.\]
Therefore, if \eqref{hil} holds true, the optimal constant must satisfy $C(m,1,R)\simeq mR^s$.
\item[$(b)$] If $R^{\frac{s}{p-1}}\lesssim m\lesssim 1$, then we let $\mu-m\simeq m^{2-p}R^{s}$. Then $\mu\gtrsim m$ and, since $p-2<0$, we have in $D_R(x_0)$
\[\fpl u \gtrsim -m^{2-p}\mu^{p-2} \gtrsim -1.\]
As before, this implies $C(m,1,R)\simeq m^{2-p}R^s$.
\item[$(c)$] Finally, if $m\lesssim R^{\frac{s}{p-1}}\lesssim 1$, then we choose $\mu-m\simeq R^{\frac{s}{p-1}}$, so that $\mu\gtrsim R^{\frac{s}{p-1}}$ and hence in $D_R(x_0)$
\[\fpl u \gtrsim - R^{-s} R^{\frac{s}{p-1}}\mu^{p-2} \gtrsim -1\]
In this case we thus have $C(m,1,R)\simeq R^{\frac{s}{p-1}}$.
\end{itemize}
In conclusion, to cover these three possible scenarios, we expect the constant $C(m,1,R)$ in \eqref{hil}  to have the form
\[C(m,1,R) \simeq R^{\frac{s}{p-1}}+m^{2-p} R^s +mR^s.\]
\end{example}

\noindent
For simplicity we will assume henceforth $0\in\partial\Omega$ and we choose $0$ as the center of balls. We must distinguish two cases according to the size of the excess (defined in \eqref{lum}), so we prefer to use different symbols in the differential inequalities to keep track of the differences.
\vskip2pt
\noindent
We begin with a lower bound for supersolutions with large excess:

\begin{lemma}\label{lbl}
Let $R>0$ be small enough depending on $N$, $p$, $s$, and $\Omega$, and let $u\in\widetilde{W}^{s,p}(D_R)$, $m,K> 0$ satisfy
\beq\label{lbl1}
\begin{cases}
\fpl u \ge -K & \text{in $D_R$} \\
u \ge m\ds & \text{in $\R^N$.}
\end{cases}
\eeq
Then, there exist $\gamma_1,C_1>1>\sigma_1>0$, depending on $N$, $p$ and $s$, s.t.\ if $L(u,m,R)\ge m\gamma_1$ then
\beq\label{lbl2}
\inf_{D_{R/2}}\Big(\frac{u}{\ds}-m\Big) \ge \sigma_1L(u,m,R)-C_1(KR^s)^\frac{1}{p-1}.
\eeq
\end{lemma}
\begin{proof}
Let $R\in(0,\rho_\Omega/4)$, with $\rho_\Omega$ defined as in Section \ref{sec2}. We define the regularized domain $A_R$ as in \eqref{ard} (with $x_0=0$). Consider the torsion problem
\beq\label{lbl3}
\begin{cases}
\fpl v = 1 & \text{in $A_R$} \\
v = 0 & \text{in $A_R^c$.}
\end{cases}
\eeq
By Proposition \ref{top} \ref{top1}, the solution $v\in W^{s,p}_0(A_R)$ of \eqref{lbl3} satisfies for all $x\in\R^N$
\[v(x) \ge \frac{1}{C}\rho_{A_R}^\frac{s}{p-1}{\rm d}_{A_R}^s(x)\]
for $C=C(N, p, s)>0$.
It has already been pointed out that $\rho_{A_R}\ge R/16$. By Lemma \ref{arp} \ref{arp2},   ${\rm d}_\Omega\le 6\, {\rm d}_{A_R}$ in $D_{R/2}$, so we can find $c=c(N,p,s)>0$ s.t.\ for a.e.\ $x\in D_{R/2}$
\beq\label{lbl4}
v(x) \ge cR^\frac{s}{p-1}\ds(x).
\eeq
Besides, by Proposition \ref{top} \ref{top2} we have in all of $\R^N$
\[\fpl v \le 1.\]
Fix $\lambda>0$ (to be determined later) and set for all $x\in\R^N$
\[w_\lambda(x) = \frac{\lambda}{R^\frac{s}{p-1}}v(x)+\chi_{\tilde B_R}(x)u(x).\]
By Proposition \ref{spp} and the inequality above, we have $w_\lambda\in\widetilde{W}^{s,p}(D_R)$ and for all $x\in D_R$
\[\fpl w_\lambda(x) \le \frac{\lambda^{p-1}}{R^s}+2\int_{\tilde B_R}\frac{(w_\lambda(x)-u(y))^{p-1}-w_\lambda^{p-1}(x)}{|x-y|^{N+ps}}\,dy.\]
We need to estimate the integrand above. To this purpose, we use Proposition \ref{top} \ref{top3} and the construction of $w_\lambda$ to see that for all $x\in D_R$
\[w_\lambda(x) = \frac{\lambda}{R^\frac{s}{p-1}}v(x) \le C\lambda R^s.\]
Now fix $x\in D_R$, $y\in\tilde B_R$. Using \eqref{ei1} with $a=w_\lambda(x)$, $b=u(y)$ we get
\begin{align*}
(w_\lambda(x)-u(y))^{p-1}-w_\lambda^{p-1}(x) &\le w_\lambda^{p-1}(x)-u^{p-1}(y) \\
&\le C\lambda^{p-1}R^{(p-1)s}-(u(y)-m\ds(y))^{p-1}.
\end{align*}
Besides we have $R/2\le|x-y|\le 3R$ and $3R/2\le{\rm d}_\Omega (y)\le 2R$ (see \eqref{geo}), so continuing from the estimate of $\fpl w_\lambda(x)$ and recalling \eqref{lbl1} we get for all $x\in D_R$
\begin{align*}
\fpl w_\lambda(x) &\le \frac{\lambda^{p-1}}{R^s}+C\int_{\tilde B_R}\frac{(\lambda R^s)^{p-1}}{|x-y|^{N+ps}}\,dy-\frac{1}{C}\int_{\tilde B_R}\frac{(u(y)-m\ds(y))^{p-1}}{|x-y|^{N+ps}}\,dy \\
&\le \frac{\lambda^{p-1}}{R^s}+C\frac{\lambda^{p-1}}{R^{N+s}}|\tilde B_R|-\frac{1}{C}\frac{R^{(p-1)s}}{R^{N+ps}}\int_{\tilde B_R}\Big(\frac{u(y)}{\ds(y)}-m\Big)^{p-1}\,dy \\
&\le (C+1)\frac{\lambda^{p-1}}{R^s}-\frac{1}{CR^s}\dashint_{\tilde B_R}\Big(\frac{u(y)}{\ds(y)}-m\Big)^{p-1}\,dy \\
&\le \Big[(C+1)\lambda^{p-1}-\frac{L(u,m,R)^{p-1}}{C}\Big]\frac{1}{R^s},
\end{align*}
with $C=C(N,p,s)>1$. Note that $\lambda>0$ is arbitrary so far. Now fix it s.t.\
\[(C+1)\lambda^{p-1} = \frac{L(u,m,R)^{p-1}}{2C},\]
so we have for all $x\in D_R$
\beq\label{lbl5}
\fpl w_\lambda(x) \le -\frac{L(u,m,R)^{p-1}}{2CR^s}.
\eeq
Let $c>0$ be the constant in \eqref{lbl4}, and set
\[\sigma_1 =\frac{1}{\gamma_1} = \frac{c}{2(2C^2+C)^\frac{1}{p-1}}, \qquad C_1 = \sigma_1(2C^2+C)^\frac{1}{p-1}.\]
So we have $C_1,\gamma_1>1>\sigma_1>0$, depending on $N$, $p$, $s$. We distinguish two cases:
\begin{itemize}[leftmargin=1cm]
\item[$(a)$] If $L(u,m,R)<(2CKR^s)^\frac{1}{p-1}$, then by choice of $\sigma_1$, $C_1$ the right hand side of \eqref{lbl1} is negative, and hence for all $x\in D_{R/2}$
\[\frac{u(x)}{\ds(x)}-m \ge 0 > \sigma_1L(u,m,R)-C_1(KR^s)^\frac{1}{p-1}\]
(even if $L(u,m,R)<m\gamma_1$).
\item[$(b)$] If $L(u,m,R)\ge(2CKR^s)^\frac{1}{p-1}$, then by \eqref{lbl5} we have
\[\begin{cases}
\fpl w_\lambda \le -K \le \fpl u & \text{in $D_R$} \\
w_\lambda = \chi_{\tilde B_R}u \le u & \text{in $D_R^c$.}
\end{cases}\]
By Proposition \ref{wcp} we have $w_\lambda\le u$ in $\R^N$. Therefore, by the choices of $\lambda$, $\sigma_1$, $\gamma_1$, and $C_1$, we get for all $x\in D_{R/2}$
\begin{align*}
u(x) &\ge \frac{\lambda}{R^\frac{s}{p-1}}v(x) \\
&\ge \frac{c}{(2C^2+C)^\frac{s}{p-1}}L(u,m,R)\ds(x) = \Big(\sigma_1+\frac{1}{\gamma_1}\Big)L(u,m,R)\ds(x).
\end{align*}
Using the assumption $L(u,m,R)\ge m\gamma_1$ we have for all $x\in D_{R/2}$
\[u(x) \ge \big(m+\sigma_1L(u,m,R)\big)\ds(x),\]
which in turn implies
\[\frac{u(x)}{\ds(x)}-m \ge \sigma_1L(u,m,R).\]
\end{itemize}
In both cases we deduce \eqref{lbl2}.
\end{proof}

\noindent
Next we prove a lower bound for supersolutions with small excess. Due to the singular nature of the operator, we get a slightly different inequality involving a free parameter $\gamma>0$ and two different positive powers of $m$:

\begin{lemma}\label{lbs}
Let $R>0$ be small enough depending on $N$, $p$, $s$, and $\Omega$, and let $u\in\widetilde{W}^{s,p}(D_R)$, $m,H> 0$ satisfy
\beq\label{lbs1}
\begin{cases}
\fpl u \ge -H & \text{in $D_R$} \\
u \ge m\ds & \text{in $\R^N$.}
\end{cases}
\eeq
Then, for all $\gamma>0$ there exist $C_\gamma>1>\sigma_\gamma>0$, depending on $N$, $p$, $s$, $\Omega$, and $\gamma$, s.t.\ if $L(u,m,R)\le m\gamma$ then
\beq\label{lbs2}
\inf_{D_{R/2}}\Big(\frac{u}{\ds}-m\Big) \ge \sigma_\gamma L(u,m,R)-C_\gamma(m+Hm^{2-p})R^s.
\eeq
\end{lemma}
\begin{proof}
Fix $\varphi\in C^\infty_c(B_1)$ s.t.\ $0\le\varphi\le 1$ in $\R^N$ and $\varphi=1$ in $B_{1/2}$. Let $\rho_\Omega>0$ be defined as in Section \ref{sec2} and $\rho'_\Omega>0$ be as in Proposition \ref{bar} (with $U=\Omega$), then fix $R$ satisfying
\[0 < R < \min\Big\{\frac{\rho_\Omega}{8},\,\rho'_\Omega\Big\}.\]
Also fix $\lambda>0$ (to be determined later) and set for all $x\in\R^N$
\[v_\lambda(x) = m\Big(1+\lambda\varphi\Big(\frac{x}{R}\Big)\Big)\ds(x).\]
By Proposition \ref{bar}  there exist $\lambda_0>0$ and $C>0$ (both depending on $N$, $p$, $s$, $\Omega$, and $\varphi$) s.t.\ for all $\lambda\in(0,\lambda_0]$ we have for all $x\in D_R$
\[\fpl v_\lambda(x) \le Cm^{p-1}\Big(1+\frac{\lambda}{R^s}\Big).\]
Further, set for all $x\in\R^N$
\[w_\lambda(x) = \begin{cases}
v_\lambda(x) & \text{if $x\in\tilde B_R^c$} \\
u(x) & \text{if $x\in\tilde B_R$.}
\end{cases}\]
By Proposition \ref{spp} we have $w_\lambda\in\widetilde{W}^{s,p}(D_R)$ and for all $x\in D_R$
\begin{align}\label{lbs3}
\fpl &w_\lambda(x) = \fpl v_\lambda(x)+2\int_{\tilde B_R}\frac{(v_\lambda(x)-u(y))^{p-1}-(v_\lambda(x)-v_\lambda(y))^{p-1}}{|x-y|^{N+ps}}\,dy \\
\nonumber&\le Cm^{p-1}\Big(1+\frac{\lambda}{R^s}\Big)-2\int_{\tilde B_R}\frac{(u(y)-v_\lambda(x))^{p-1}-(v_\lambda(y)-v_\lambda(x))^{p-1}}{|x-y|^{N+ps}}\,dy.
\end{align}
This time estimating the integrand requires some more labor than in Lemma \ref{lbl}. Set
\[\lambda_0' = \min\Big\{\lambda_0,\,\frac{1}{2}\Big(\frac{3^s}{2^s}-1\Big)\Big\} > 0,\]
depending on $N$, $p$, $s$, and $\Omega$. Assume from now on $\lambda\in(0,\lambda_0']$, and fix $x\in D_R$, $y\in\tilde B_R$. By \eqref{lbs1} we have
\[u(y) \ge m\ds(y) = v_\lambda(y).\]
By \eqref{geo} we have
\[\inf_{\tilde{B}_R} v_\lambda=\inf_{\tilde{B}_R} m \ds \ge m \frac{3^s}{2^s} R^s,\]
as well as
\[\sup_{D_R} v_\lambda\le \sup_{D_R} m (1+\lambda_0')  \ds \le m \left(\frac{3^s}{2^{s+1}}+\frac{1}{2}\right)R^s,\]
which imply
\begin{align*}
u(y)-v_\lambda(x) &\ge v_\lambda(y)-v_\lambda(x) \\
&\ge m\Big(\frac{3^s}{2^s}-1\Big)\frac{R^s}{2} > 0.
\end{align*}
Now, for all $x\in D_R$, $y\in\tilde B_R$ set $a=u(y)-v_\lambda(x)$, $b=v_\lambda(y)-v_\lambda(x)\le m(2R)^s$ (both non-negative). By Lagrange's theorem and monotonicity of $t\mapsto t^{p-2}$ in $(0,\infty)$ we have
\begin{align*}
(u(y)-v_\lambda(x))^{p-1}-(v_\lambda(y)-v_\lambda(x))^{p-1} &\ge (p-1)\min_{a\le t\le b}|t|^{p-2}(u(y)-v_\lambda(y)) \\
&= (p-1)\frac{u(y)-v_\lambda(y)}{(u(y)-v_\lambda(x))^{2-p}}.
\end{align*}
Further, apply inequality \eqref{ei2} with an arbitrary $\theta>0$ (to be determined later) to get
\begin{align*}
(u(y)-v_\lambda(x))^{p-1}-&(v_\lambda(y)-v_\lambda(x))^{p-1} \\
&\ge  \frac{p-1}{(\theta+1)^{2-p}}\big[\theta^{2-p}(u(y)-v_\lambda(y))^{p-1}-\theta(v_\lambda(y)-v_\lambda(x))^{p-1}\big] \\
&\ge \frac{p-1}{(\theta+1)^{2-p}}\big[\theta^{2-p}(u(y)-v_\lambda(y))^{p-1}-\theta m^{p-1}(2R)^{(p-1)s}\big].
\end{align*}
Besides, by \eqref{geo} we have $R/2\le|x-y|\le 3R$ for all $x\in D_R$, $y\in \tilde{B}_R$. Therefore, continuing from \eqref{lbs3}, we have for all $x\in D_R$
\begin{align*}
\fpl w_\lambda(x) &\le Cm^{p-1}\Big(1+\frac{\lambda}{R^s}\Big)-\frac{2(p-1)}{(\theta+1)^{2-p}}\int_{\tilde B_R}\hspace{-3pt}\frac{\theta^{2-p}(u(y)-v_\lambda(y))^{p-1}-\theta (2^sR^sm)^{p-1}}{|x-y|^{N+ps}}dy \\
&\le  Cm^{p-1}\Big(1+\frac{\lambda}{R^s}\Big)-\frac{1}{C}\Big(\frac{\theta}{\theta+1}\Big)^{2-p}\frac{R^{(p-1)s}}{R^{N+ps}}\dashint_{\tilde B_R}\Big(\frac{u(y)}{\ds(y)}-m\Big)^{p-1}\,dy \\
& \quad +\frac{C\theta m^{p-1}}{(\theta+1)^{2-p}}\frac{R^{(p-1)s}}{R^{N+ps}}|\tilde B_R| \\
&\le Cm^{p-1}+\Big[C(\lambda+\theta)m^{p-1}-\frac{1}{C}\Big(\frac{\theta}{\theta+1}\Big)^{2-p}L(u,m,R)^{p-1}\Big]\frac{1}{R^s},
\end{align*}
with $C=C(N,p,s,\Omega)>1$ and $\theta>0$ still to be chosen. Now let $\gamma>0$ come into play, and assume $L(u,m,R)\le m\gamma$. Since $p\in(1,2)$, we can find $\delta\in(0,1)$ (depending on $N$, $p$, $s$, $\Omega$, and $\gamma$) s.t.\
\[C\delta \le \frac{\delta^{2-p}}{2C(\gamma+1)^{2-p}}.\]
Pick such a $\delta$ and set
\[\theta = \frac{\delta L(u,m,R)}{m}.\]
Clearly we have $\theta\in(0,\gamma)$. By the relations above,
\begin{align*}
C\theta m^{p-1} &= C\delta\frac{L(u,m,R)}{m^{2-p}} \\
&\le \frac{\delta^{2-p}}{2C(\gamma+1)^{2-p}}\,\frac{L(u,m,R)}{m^{2-p}} \\
&\le \frac{1}{2C}\Big(\frac{\theta}{\gamma+1}\Big)^{2-p}L(u,m,R)^{p-1}.
\end{align*}
Plugging this into the estimate of $\fpl w_\lambda$, we have for all $x\in D_R$
\[\fpl w_\lambda(x) \le Cm^{p-1}+\Big[C\lambda m^{p-1}-\frac{1}{2C}\Big(\frac{\theta}{\gamma+1}\Big)^{2-p}L(u,m,R)^{p-1}\Big]\frac{1}{R^s}.\]
Set for simplicity
\[\kappa = \frac{\delta^{2-p}}{2C(\gamma+1)^{2-p}} \in (0,1).\]
Then, the last term of the inequality above rephrases as follows, adjusting the exponent of the excess:
\begin{align*}
\frac{1}{2C}\Big(\frac{\theta}{\gamma+1}\Big)^{2-p}L(u,m,R)^{p-1} &= \frac{\kappa\theta^{2-p}}{\delta^{2-p}}L(u,m,R)^{p-1} \\
&= \frac{\kappa}{m^{2-p}}L(u,m,R).
\end{align*}
Therefore we have for all $\lambda\in(0,\lambda_0']$ and all $x\in D_R$
\beq\label{lbs4}
\fpl w_\lambda(x) \le Cm^{p-1}+\Big[C\lambda m^{p-1}-\frac{\kappa}{m^{2-p}}L(u,m,R)\Big]\frac{1}{R^s}.
\eeq
We can now establish the constants appearing in the conclusion:
\[\sigma_\gamma = \min\Big\{\frac{\lambda'_0}{\gamma},\,\frac{\kappa}{2C}\Big\}, \qquad C_\gamma = \frac{2C\sigma_\gamma}{\kappa},\]
so that $C_\gamma>1>\sigma_\gamma>0$ and both depend on $N$, $p$, $s$, $\Omega$, and $\gamma$. Also set
\[\lambda = \frac{\sigma_\gamma}{m}L(u,m,R).\]
By assumption $L(u,m,R)\le m\gamma$ and definition of $\sigma_\gamma$ we have $\lambda\in(0,\lambda_0']$. Besides, the second term in \eqref{lbs4}, for this choice of $\lambda$, satisfies
\[C\lambda m^{p-1} \le \frac{\kappa}{2m^{2-p}}L(u,m,R).\]
Summarizing, we have for all $x\in D_R$
\beq\label{lbs5}
\fpl w_\lambda(x) \le Cm^{p-1}-\frac{\kappa}{2m^{2-p}}\frac{L(u,m,R)}{R^s}.
\eeq
Now we distinguish two cases (in which we let $H>0$ be as in \eqref{lbs1}):
\begin{itemize}[leftmargin=1cm]
\item[$(a)$] If
\[L(u,m,R) \ge \frac{2}{\kappa}(Cm+Hm^{2-p})R^s,\]
then by \eqref{lbs5} we have for all $x\in D_R$
\[\fpl w_\lambda(x) \le Cm^{p-1}-\frac{Cm+Hm^{2-p}}{m^{2-p}} = -H,\]
while for all $x\in D_R^c$
\[w_\lambda(x) = \begin{cases}
u(x) & \text{if $x\in\tilde B_R$} \\
m\ds(x) & \text{if $x\in\tilde B_R^c$.}
\end{cases}\]
Thus, by \eqref{lbs1} we have
\[\begin{cases}
\fpl w_\lambda \le \fpl u & \text{in $D_R$} \\
w_\lambda \le u & \text{in $D_R^c$.}
\end{cases}\]
Proposition \ref{wcp} now implies $w_\lambda\le u$ in $\R^N$, in particular for all $x\in D_{R/2}$
\[
\frac{u(x)}{\ds(x)}-m \ge \frac{w_\lambda(x)}{\ds(x)}-m = \lambda m = \sigma_\gamma L(u,m,R).
\]
\item[$(b)$] If on the contrary
\[L(u,m,R) < \frac{2}{\kappa}(Cm+Hm^{2-p})R^s,\]
then by the choice of $\sigma_\gamma$, $C_\gamma$ we have
\begin{align*}
\sigma_\gamma L(u,m,R)-C_\gamma(m+Hm^{2-p})R^s &\le \frac{2\sigma_\gamma}{\kappa}(Cm+Hm^{2-p})R^s-\frac{2\sigma_\gamma}{\kappa}C(m+Hm^{2-p})R^s \\
&= \frac{2\sigma_\gamma}{\kappa}(1-C)Hm^{2-p}R^s < 0.
\end{align*}
Thus, \eqref{lbs2} trivially holds since its right hand side is negative.
\end{itemize}
In both cases, we deduce \eqref{lbs2}.
\end{proof}

\noindent
We can now prove our lower bound for subsolutions of fractional $p$-Laplacian equations. We will use a right hand side of the type $-\min\{K,H\}$, which might seem redundant since one could equivalently take $H=K$. The reason for such a choice is instrumental to the proof of the oscillation estimate on $u/\ds$, where for suitable truncations of $u$ we will compute two different lower bounds for their fractional $p$-Laplacians:

\begin{proposition}\label{lba}
Let $R>0$ be small enough depending on $N$, $p$, $s$, and $\Omega$, and let $u\in\widetilde{W}^{s,p}(D_R)$, $m,K,H> 0$ satisfy
\beq\label{lba1}
\begin{cases}
\fpl u \ge -\min\{K,H\} & \text{in $D_R$} \\
u \ge m\ds & \text{in $\R^N$.}
\end{cases}
\eeq
Then, there exist $C>1>\sigma>0$, depending on $N$, $p$, $s$, and $\Omega$, s.t.\
\beq\label{lba2}
\inf_{D_{R/2}}\Big(\frac{u}{\ds}-m\Big) \ge \sigma L(u,m,R)-C(KR^s)^\frac{1}{p-1}-C(m+Hm^{2-p})R^s.
\eeq
\end{proposition}
\begin{proof}
Let $R>0$ be small enough s.t.\ both Lemma \ref{lbl} and Lemma \ref{lbs} apply. By \eqref{lba1} we see that $u$ satisfies \eqref{lbl1}. Thus, by Lemma \ref{lbl} there exist $\gamma_1,C_1>1>\sigma_1>0$ s.t.\ if $L(u,m,R)\ge m\gamma_1$ then
\[\inf_{D_{R/2}}\Big(\frac{u}{\ds}-m\Big) \ge \sigma_1 L(u,m,R)-C_1(KR^s)^\frac{1}{p-1}.\]
Besides, $u$ also satisfies \eqref{lbs1}. Thus, by Lemma \ref{lbs} with $\gamma=\gamma_1$ there exist $C_{\gamma_1}>1>\sigma_{\gamma_1}>0$ s.t.\ if $L(u,m,R)\le m\gamma_1$ then
\[\inf_{D_{R/2}}\Big(\frac{u}{\ds}-m\Big) \ge \sigma_{\gamma_1} L(u,m,R)-C_{\gamma_1}(m+Hm^{2-p})R^s.\]
All constants depend on $N$, $p$, $s$, and $\Omega$. Now set
\[\sigma = \min\{\sigma_1,\,\sigma_{\gamma_1}\}, \ C = \max\{C_1,\,C_{\gamma_1}\}.\]
Then, $C>1>\sigma>0$ depend on $N$, $p$, $s$, and $\Omega$ and \eqref{lba2} follows from either the first or the second of the bounds above, according to the value of the excess $L(u,m,R)\ge 0$.
\end{proof}

\section{Upper bound}\label{sec4}

\noindent
In this section we prove an upper bound for subsolutions of \eqref{dir}-type problems in domains of the type $D_R$, globally bounded from above by $M\ds$ ($M>0$). This bound is equivalent to the weak Harnack inequality \eqref{hiu} stated in Section \ref{sec1}, with the constant $C(M,K,R)$ taking the same form as in Example \ref{cmr}. Again we assume $0\in\partial\Omega$ and center balls at $0$, and we distinguish between subsolutions with large and small excess, respectively.
\vskip2pt
\noindent
We begin with a local negativity property, that is, subsolutions with large excess are in fact negative on a smaller set:

\begin{lemma}\label{ubn}
Let $R>0$ be small enough depending on $N$, $p$, $s$, and $\Omega$, and let $u\in\widetilde{W}^{s,p}(D_R)$, $M,K>0$ satisfy
\beq\label{ubn1}
\begin{cases}
\fpl u \le K & \text{in $D_R$} \\
u \le M\ds & \text{in $\R^N$}.
\end{cases}
\eeq
Then, there exists $\tilde C_2>1$, depending on $N$, $p$, $s$, and $\Omega$ s.t.\ if
\[L(u,M,R) \ge \tilde C_2\big(M+(KR^s)^\frac{1}{p-1}\big),\]
then $u\le 0$ in $D_{R/2}$.
\end{lemma}
\begin{proof}
First let $R\in (0,\rho_\Omega/4)$, then fix $x_0\in D_{R/2}$ (possibly excluding a subset with zero measure). By Proposition \ref{ptz}, there exist $v\in W^{s,p}_0(\Omega)\cap C^0(\R^N)$, $C=C(N,p,s,\Omega)>1$ satisfying $v(x_0)=0$ and
\beq\label{ubn2}
\begin{cases}
\displaystyle|\fpl v| \le \frac{C}{R^s} & \text{in $D_{2R}$} \\
|v| \le CR^s & \text{in $D_{2R}$} \\
\displaystyle v \ge \frac{\ds}{C} & \text{in $D_R^c$.}
\end{cases}
\eeq
Comparing \eqref{ubn1} and \eqref{ubn2} we get for all $x\in D_R^c$
\[u(x) \le M\ds(x) \le CMv(x).\]
Now set for all $x\in\R^N$
\[w(x) = \begin{cases}
CMv(x) & \text{if $x\in\tilde B_R^c$} \\
u(x) & \text{if $x\in\tilde B_R$.}
\end{cases}\]
By Proposition \ref{spp} we have $w\in\widetilde{W}^{s,p}(D_R)$ and for all $x\in D_R$
\begin{align*}
\fpl w(x) &= (CM)^{p-1}\fpl v(x)\\
& \ +2\int_{\tilde B_R}\frac{(CMv(x)-u(y))^{p-1}-(CMv(x)-CMv(y))^{p-1}}{|x-y|^{N+ps}}\,dy.
\end{align*}
We need to estimate the integrand. Fix $x\in D_R$, $y\in\tilde B_R$. By \eqref{ei3} with 
\[a=CMv(x)-CMv(y), \qquad b=CMv(y)-u(y)\ge 0\]
and the second relation of \eqref{ubn2}, we have
\begin{align*}
(CMv(x)-u(y))^{p-1}&-(CMv(x)-CMv(y))^{p-1} \\
&\ge (CMv(y)-u(y))^{p-1}-|CMv(x)-CMv(y)|^{p-1} \\
&\ge (CMv(y)-u(y))^{p-1}-C'M^{p-1}R^{(p-1)s},
\end{align*}
with both $C,C'>0$ depending on $N$, $p$, $s$, and $\Omega$. Recalling \eqref{geo}, for all $x\in D_R$, $y\in\tilde B_R$ we have $R/2\le |x-y|\le 3R$ and $R/2\le{\rm d}_\Omega(y)\le 2R$. Therefore, using also the third inequality in \eqref{ubn2}, for all $x\in D_R$ we have
\begin{align*}
\int_{\tilde B_R}&\frac{(CMv(x)-u(y))^{p-1}-(CMv(x)-CMv(y))^{p-1}}{|x-y|^{N+ps}}\,dy \\
& \ \ge \frac{1}{CR^{N+ps}}\int_{\tilde B_R}(M\ds(y)-u(y))^{p-1}\,dy-\frac{C'M^{p-1}}{R^{N+s}}|\tilde B_R| \\
& \ \ge \frac{R^{N+(p-1)s}}{CR^{N+ps}}\,\dashint_{\tilde B_R}\Big(M-\frac{u(y)}{\ds(y)}\Big)^{p-1}\,dy-\frac{C'M^{p-1}}{R^s} \\
& \ \ge \frac{L(u,M,R)^{p-1}}{CR^s}-\frac{C'M^{p-1}}{R^s}.
\end{align*}
Plugging the last inequality into the estimate of $\fpl w$ and using the first relation of \eqref{ubn2}, we get for all $x\in D_R$
\[\fpl w(x) \ge \frac{L(u,M,R)^{p-1}}{CR^s}-\frac{CM^{p-1}}{R^s},\]
for a suitable $C=C(N,p,s,\Omega)>1$. Clearly we can find $\tilde C_2=\tilde C_2(N,p,s,\Omega)>1$ s.t.\
\[\tilde C_2\big(M+(KR^s)^\frac{1}{p-1}\big) \ge \big(C^2M^{p-1}+CKR^s\big)^\frac{1}{p-1},\]
with $K>0$ as in \eqref{ubn1}. Now assume
\[L(u,M,R) \ge \tilde C_2\big(M+(KR^s)^\frac{1}{p-1}\big).\]
Then, for all $x\in D_R$ we get
\[\fpl w(x) \ge \frac{C^2M^{p-1}+CKR^s}{CR^s}-\frac{CM^{p-1}}{R^s} = K.\]
Summarizing, by \eqref{ubn1} and the definition of $w$ we have
\[\begin{cases}
\fpl u \le K \le \fpl w & \text{in $D_R$} \\
u \le w & \text{in $D_R^c$.}
\end{cases}\]
Proposition \ref{wcp} now implies $u\le w$ in all of $\R^N$. In particular we have
\[u(x_0) \le CMv(x_0) = 0,\]
and conclude.
\end{proof}

\noindent
We can now prove the upper bound for subsolutions with large excess, partially analogous to Lemma \ref{lbl} above. Note that, due to the different approach followed here, the upper bound only holds in the smaller set $D_{R/4}$.

\begin{lemma}\label{ubl}
Let $R>0$ be small enough depending on $N$, $p$, $s$, and $\Omega$, and let $u\in\widetilde{W}^{s,p}(D_R)$, $M,K>0$ satisfy \eqref{ubn1}. Then, there exist $\gamma_2,C_2>1>\sigma_2>0$, depending on $N$, $p$, $s$, and $\Omega$, s.t.\ if $L(u,M,R)\ge M\gamma_2$ then
\beq\label{ubl1}
\inf_{D_{R/4}}\Big(M-\frac{u}{\ds}\Big) \ge \sigma_2L(u,M,R)-C_2(KR^s)^\frac{1}{p-1}.
\eeq
\end{lemma}
\begin{proof}
Fix $R\in(0,\rho_\Omega/4)$, and let $\tilde C_2>1$ be as in Lemma \ref{ubn}. Fix $\gamma_2,C_2>1>\sigma_2>0$ (to be determined later) s.t.\
\beq\label{ubl2}
\min\Big\{\frac{C_2}{\sigma_2},\,\gamma_2\Big\} \ge 2\tilde C_2.
\eeq
Assume from the start $L(u,M,R)\ge M\gamma_2$ and
\beq\label{ubl3}
\sigma_2L(u,M,R) \ge C_2(KR^s)^\frac{1}{p-1},
\eeq
otherwise the conclusion is trivial due to \eqref{ubn1} (the right hand side of \eqref{ubl1} becomes negative). Therefore, we have
\begin{align*}
L(u,M,R) &\ge \frac{M\gamma_2}{2}+\frac{C_2}{2\sigma_2}(KR^s)^\frac{1}{p-1} \\
&\ge \tilde C_2\big(M+(KR^s)^\frac{1}{p-1}\big).
\end{align*}
By Lemma \ref{ubn}, we have $u\le 0$ in $D_{R/2}$. We define $A_{R/2}$ as in \eqref{ard} (centered at $0$), so by Lemma \ref{arp} (note that $R<\rho_\Omega$) we have $D_{3R/8}\subseteq A_{R/2}\subseteq D_{R/2}$, and $6\, {\rm d}_{A_{R/2}}\ge {\rm d_\Omega}$ in $D_{R/4}$. Next, let $\varphi\in W^{s,p}_0(A_{R/2})$ be the solution of the torsion problem
\beq\label{ubl4}
\begin{cases}
\fpl\varphi = 1 & \text{in $A_{R/2}$} \\
\varphi = 0 & \text{in $A_{R/2}^c$.}
\end{cases}
\eeq
By Proposition \ref{top} \ref{top2} we have $\fpl\varphi\le 1$ in all of $\R^N$, while Proposition \ref{top} \ref{top1} \ref{top3} and the relations above imply for all $x\in D_{R/4}$
\[\frac{R^\frac{s}{p-1}}{C}\ds(x) \le \varphi(x) \le CR^s,\]
with $C=C(N,p,s)>1$. Now fix $\lambda>0$ (to be determined later) and set for all $x\in\R^N$
\[v_\lambda(x) = \begin{cases}
-\lambda R^{-\frac{s}{p-1}}\varphi(x) & \text{if $x\in D_{R/2}$} \\
M\ds(x) & \text{if $x\in D_{R/2}^c$.}
\end{cases}\]
Reasoning as in \cite[eq.\ (4.19)]{IMS1} (an argument which holds for any $p>1$) and exploiting \eqref{ubl4}, we have for all $x	\in A_{R/2}$
\beq\label{ubl5}
\fpl v_\lambda(x) \ge -\frac{C}{R^s}(\lambda^{p-1}+M^{p-1}),
\eeq
with $C=C(N,p,s,\Omega)>1$. Further set for all $x\in\R^N$
\[w_\lambda(x) = \begin{cases}
v_\lambda(x) & \text{if $x\in\tilde B_R^c$} \\
u(x) & \text{if $x\in\tilde B_R$.}
\end{cases}\]
By Proposition \ref{spp} we have $w_\lambda\in\widetilde{W}^{s,p}(A_{R/2})$ and, using also \eqref{ubl5} and the definition of $v_\lambda$, we have for all $x\in A_{R/2}$
\begin{align*}
\fpl w_\lambda(x) &= \fpl v_\lambda(x)+2\int_{\tilde B_R}\frac{(v_\lambda(x)-u(y))^{p-1}-(v_\lambda(x)-v_\lambda(y))^{p-1}}{|x-y|^{N+ps}}\,dy \\
&\ge 2\int_{\tilde B_R}\frac{(v_\lambda(x)-u(y))^{p-1}-(v_\lambda(x)-M\ds(y))^{p-1}}{|x-y|^{N+ps}}\,dy-\frac{C}{R^s}(\lambda^{p-1}+M^{p-1}).
\end{align*}
As in previous cases, we will now estimate the integrand. By the properties of $\varphi$, by taking $C=C(N,p,s,\Omega)>1$ even bigger if necessary, we have for all $x\in D_{R/4}$
\beq\label{ubl6}
w_\lambda(x) = -\frac{\lambda}{R^\frac{s}{p-1}}\varphi(x) \le -\frac{\lambda}{C}\ds(x).
\eeq
Besides, for all $x\in D_{R/2}$
\[|w_\lambda(x)| = |v_\lambda(x)| \le C\lambda R^s.\]
Fix $x\in A_{R/2}$, $y\in\tilde B_R$. By \eqref{ei3} with $a=v_\lambda(x)-M\ds(y)$, $b=M\ds(y)-u(y)\ge 0$ (see \eqref{ubn1}), we have
\begin{align*}
(v_\lambda(x)-u(y))^{p-1}\hspace{-1pt}-(v_\lambda(x)-M\ds(y))^{p-1} &\ge (M\ds(y)-u(y))^{p-1}\hspace{-1pt}-|v_\lambda(x)-M\ds(y)|^{p-1} \\
&\ge (M\ds(y)-u(y))^{p-1}\hspace{-1pt}-C(\lambda^{p-1}\hspace{-1pt}+M^{p-1})R^{(p-1)s}.
\end{align*}
Taking into account the usual bounds on $|x-y|$, ${\rm d}_\Omega(y)$, and recalling the estimate above on $\fpl w_\lambda$, we have for all $x\in A_{R/2}$
\begin{align*}
\fpl w_\lambda(x) &\ge2\hspace{-1pt}\int_{\tilde B_R}\hspace{-4pt}\frac{(M\ds(y)-u(y))^{p-1}\hspace{-1pt}-C(\lambda^{p-1}\hspace{-1pt}+M^{p-1})R^{(p-1)s}}{|x-y|^{N+ps}}dy  -\frac{C}{R^s}(\lambda^{p-1}\hspace{-1pt}+M^{p-1}) \\
&\ge  \frac{R^{(p-1)s}}{CR^{N+ps}}\int_{\tilde B_R}\Big(M-\frac{u(y)}{\ds(y)}\Big)^{p-1}\,dy -C(\lambda^{p-1}\hspace{-1pt}+M^{p-1})\left[\frac{R^{(p-1)s}}{R^{N+ps}}|\tilde B_R| +\frac{1}{R^s}\right]\\
&\ge \frac{L(u,M,R)^{p-1}}{CR^s}-\frac{C}{R^s}(\lambda^{p-1}\hspace{-1pt}+M^{p-1}),
\end{align*}
for a suitable $C=C(N,p,s,\Omega)>1$. We can now fix the constants involved in the conclusion, setting
\[\gamma_2 = \max\big\{2\tilde C_2,\,(4C^2)^\frac{1}{p-1}\big\},\]
and accordingly
\[\sigma_2 = \frac{1}{C(4C^2)^\frac{1}{p-1}}, \qquad C_2 = \sigma_2\max\big\{2\tilde C_2,\,(2C)^\frac{1}{p-1}\big\},\]
so $\gamma_2,C_2>1>\sigma_2>0$ depend on $N$, $p$, $s$, and $\Omega$ and satisfy \eqref{ubl2}. Also set
\[\lambda = \frac{L(u,M,R)}{(4C^2)^\frac{1}{p-1}} > 0.\]
By \eqref{ubl3} and the assumption $L(u,M,R)\ge M\gamma_2$, we have for all $x\in A_{R/2}$
\begin{align*}
\fpl w_\lambda(x) &\ge \frac{L(u,M,R)^{p-1}}{CR^s}-\frac{C}{R^s}\Big[\frac{L(u,M,R)^{p-1}}{4C^2}+\frac{L(u,M,R)^{p-1}}{\gamma_2^{p-1}}\Big] \\
&\ge \frac{L(u,M,R)^{p-1}}{2CR^s} \\
&\ge \frac{C_2^{p-1}KR^s}{2C\sigma_2^{p-1}R^s} \ge K.
\end{align*}
Besides, for all $x\in A_{R/2}^c$ we distinguish three cases:
\begin{itemize}[leftmargin=1cm]
\item[$(a)$] If $x\in D_{R/2}^c\cap\tilde B_R^c$, then by definition of $w_\lambda$ and \eqref{ubn1} we have
\[w_\lambda(x) = v_\lambda(x) = M\ds(x) \ge u(x).\]
\item[$(b)$] If $x\in\tilde B_R$, then simply
\[w_\lambda(x) = u(x).\]
\item[$(c)$] If $x\in D_{R/2}\cap A_{R/2}^c$, then by Lemma \ref{ubn}
\[w_\lambda(x) = v_\lambda(x) = -\frac{\lambda}{R^\frac{s}{p-1}}\varphi(x) = 0 \ge u(x).\]
\end{itemize}
Summarizing, we have
\[\begin{cases}
\fpl u \le K \le \fpl w_\lambda & \text{in $A_{R/2}$} \\
u \le w_\lambda & \text{in $A_{R/2}^c$.}
\end{cases}\]
By Proposition \ref{wcp}, $u\le w_\lambda$ in all of $\R^N$. Recalling \eqref{ubl6}, for all $x\in D_{R/4}$ we have
\begin{align*}
u(x) &\le -\frac{\lambda}{C}\ds(x) \\
&= -\frac{L(u,M,R)}{C(4C^2)^\frac{1}{p-1}}\,\ds(x) = -\sigma_2L(u,M,R)\ds(x).
\end{align*}
Therefore, we have for all $x\in D_{R/4}$
\[M-\frac{u(x)}{\ds(x)} \ge -\frac{u(x)}{\ds(x)} \ge \sigma_2L(u,M,R),\]
hence we deduce \eqref{ubl1}.
\end{proof}

\noindent
Next we prove the upper bound for subsolutions with small excess, analogous to Lemma \ref{lbs} above. In this case, the argument is closer to the one for lower bound:

\begin{lemma}\label{ubs}
Let $R>0$ be small enough depending on $N$, $p$, $s$, and $\Omega$, and let  $u\in\widetilde{W}^{s,p}(D_R)$, $M,H>0$ satisfy
\beq\label{ubs1}
\begin{cases}
\fpl u \le H & \text{in $D_R$} \\
u \le M\ds & \text{in $\R^N$.}
\end{cases}
\eeq
Then, for all $\gamma>1$ there exist $C_\gamma>1>\sigma_\gamma>0$, depending on $N$, $p$, $s$, $\Omega$, and $\gamma$, s.t.\ if $L(u,M,R)\le M\gamma$ then
\beq\label{ubs2}
\inf_{D_{R/2}}\Big(M-\frac{u}{\ds}\Big) \ge \sigma_\gamma L(u,M,R)-C_\gamma(M+HM^{2-p})R^s.
\eeq
\end{lemma}
\begin{proof}
Since the argument closely follows that of Lemma \ref{lbs}, we sketch it quickly. Let $R>0$ satisfy
\[R < \min\Big\{\frac{\rho_\Omega}{8},\,\rho'_\Omega\Big\}\]
($\rho'_\Omega>0$ being defined as in Proposition \ref{bar}). We define $\varphi\in C^\infty_c(B_1)$ as in Lemma \ref{lbs}, then we fix $\lambda<0$ and set for all $x\in\R^N$
\[v_\lambda(x) = M\Big(1+\lambda\varphi\Big(\frac{x}{R}\Big)\Big)\ds(x),\]
so $v_\lambda\in\widetilde{W}^{s,p}(D_R)$. By Proposition \ref{bar}, there exists $\lambda_0>0$ s.t.\ whenever $\lambda\in[-\lambda_0,0)$, we have for all $x\in D_R$
\[\fpl v_\lambda(x) \ge -CM^{p-1}\Big(1+\frac{|\lambda|}{R^s}\Big),\]
with $C>1$, $\lambda_0>0$ depending on $N$, $p$, $s$, and $\Omega$. Next we define $w_\lambda\in\widetilde{W}^{s,p}(D_R)$ as in Lemma \ref{ubl} and get for all $x\in D_R$
\beq\label{ubs3}
\fpl w_\lambda(x) \ge 2\int_{\tilde B_R}\hspace{-4pt}\frac{(v_\lambda(y)-v_\lambda(x))^{p-1}\hspace{-1pt}-(u(y)-v_\lambda(x))^{p-1}}{|x-y|^{N+ps}} dy-CM^{p-1}\Big(1+\frac{|\lambda|}{R^s}\Big).
\eeq
This time the estimate of the integrand must be performed in two different ways, due to the singularity of $t\mapsto|t|^{p-2}$ at $0$. Fix $x\in D_R$, $y\in\tilde B_R$, and distinguish two cases:
\begin{itemize}[leftmargin=1cm]
\item[$(a)$] If $u(y)\ge v_\lambda(x)$, then we argue as in Lemma \ref{lbs}. We set
\[a=v_\lambda(y)-v_\lambda(x), \quad b=u(y)-v_\lambda(x),\]
so $a,b\ge 0$ and $a\ge b$ thanks to $v_\lambda(y)=Md_\Omega^s(y)\ge u(y)$. Then we use Lagrange's theorem and monotonicity of $t\mapsto t^{p-2}$ in $(0,\infty)$ to get
\begin{align*}
(v_\lambda(y)-v_\lambda(x))^{p-1}-(u(y)-v_\lambda(x))^{p-1}
&\ge (p-1)\min_{b\le t\le a}|t|^{p-2}(v_\lambda(y)-u(y)) \\
&= (p-1)\frac{v_\lambda(y)-u(y)}{(v_\lambda(y)-v_\lambda(x))^{2-p}}.
\end{align*}
Further, by inequality \eqref{ei2} with an arbitrary $\theta>0$, we get
\begin{align*}
(v_\lambda(y)-v_\lambda(x))^{p-1}&-(u(y)-v_\lambda(x))^{p-1} \\
& \ \ge \frac{p-1}{(\theta+1)^{2-p}}\big[\theta^{2-p}(v_\lambda(y)-u(y))^{p-1}-\theta(v_\lambda(y)-v_\lambda(x))^{p-1}\big] \\
& \ \ge \frac{p-1}{(\theta+1)^{2-p}}\big[\theta^{2-p}(v_\lambda(y)-u(y))^{p-1}-\theta(MR^s)^{p-1}\big].
\end{align*}
\item[$(b)$] If $u(y)<v_\lambda(x)$, note that by \eqref{ubs1} and \eqref{geo}
\[v_\lambda(x) \le M\ds(x) \le MR^s \le M\ds(y) = v_\lambda(y).\]
Then, by subadditivity of $t\mapsto t^{p-1}$ in $[0,\infty)$ we have
\[(v_\lambda(y)-u(y))^{p-1}\le (v_\lambda(y)-v_\lambda(x))^{p-1}+(v_\lambda(x)-u(y))^{p-1}\]
which, since $p\in (1, 2)$, implies that for any $\theta>0$
\[(v_\lambda(y)-v_\lambda(x))^{p-1}-(u(y)-v_\lambda(x))^{p-1}\ge (p-1)\left(\frac{\theta}{\theta+1}\right)^{2-p}(v_\lambda(y)-u(y))^{p-1}.\]
\end{itemize}
All in all, for any $x\in D_R$, $y\in\tilde B_R$ we have
\[(v_\lambda(y)-v_\lambda(x))^{p-1}\hspace{-1pt}-(u(y)-v_\lambda(x))^{p-1} \ge \frac{p-1}{(\theta+1)^{2-p}}\big[\theta^{2-p}(v_\lambda(y)-u(y))^{p-1}\hspace{-1pt}-\theta (MR^s)^{p-1}\big].\]
Now we plug such estimate into \eqref{ubs3}, recall that $v_\lambda=M{\rm d}_\Omega^s$ in $\tilde{B}_R$ and find for all $x\in D_R$
\[\fpl w_\lambda(x) \ge \Big[\frac{1}{C}\Big(\frac{\theta}{\theta+1}\Big)^{2-p}L(u,M,R)^{p-1}-C(|\lambda|+\theta)M^{p-1}\Big]\frac{1}{R^s}-CM^{p-1},\]
with $C>1$, $\lambda\in[-\lambda_0,0)$, and $\theta>0$ still to be chosen. Now we fix $\gamma>1$ and assume $L(u,M,R)\le M\gamma$. As in Lemma \ref{lbs} we define $\delta,\kappa\in(0,1)$ (depending on $\gamma$), and accordingly set
\[\theta = \frac{\delta L(u,M,R)}{M} \in (0,\gamma],\]
which also satisfies
\[C\theta M^{p-1} \le \frac{1}{2C}\Big(\frac{\theta}{\theta+1}\Big)^{2-p}L(u,M,R)^{p-1}.\]
Using the above relations and setting
\[\sigma_\gamma =\min\Big\{\frac{\lambda_0}{\gamma},\,\frac{\kappa}{2C}\Big\}, \qquad C_\gamma = \frac{2C\sigma_\gamma}{\kappa},\]
we see that $C_\gamma>1>\sigma_\gamma>0$ depend on $N$, $p$, $s$, $\Omega$, and $\gamma$. Moreover, setting
\[\lambda = -\frac{\sigma_\gamma}{M}L(u,M,R) \in [-\lambda_0,0),\]
we get for all $x\in D_R$
\beq\label{ubs4}
\fpl w_\lambda(x) \ge -CM^{p-1}+\frac{\kappa}{2M^{2-p}}\,\frac{L(u,M,R)}{R^s}.
\eeq
Now, like in Lemma \ref{lbs} we let $H>0$ be as in \eqref{ubs1} and distinguish two cases according to the size of $L(u,M,R)$:
\begin{itemize}[leftmargin=1cm]
\item[$(a)$] If
\[L(u,M,R) \ge \frac{2}{\kappa}(CM+HM^{2-p})R^s,\]
then we apply \eqref{ubs4} and the definition of $w_\lambda$ to see that
\[\begin{cases}
\fpl u \le H \le \fpl w_\lambda & \text{in $D_R$} \\
u \le w_\lambda & \text{in $D_R^c$,}
\end{cases}\]
hence by Proposition \ref{wcp} $u\le w_\lambda$ in $\R^N$ and in particular for all $x\in D_R$
\[M-\frac{u(x)}{\ds(x)} \ge -M\lambda = \sigma_\gamma L(u,M,R).\]
\item[$(b)$] If
\[L(u,M,R) < \frac{2}{\kappa}(CM+HM^{2-p})R^s,\]
then from the choice of constants we have
\[\sigma_\gamma L(u,M,R) < C_\gamma(M+HM^{2-p})R^s,\]
hence the conclusion is trivial, since the right hand side in \eqref{ubs2} is negative.
\end{itemize}
In both cases we deduce \eqref{ubs2} and conclude.
\end{proof}

\noindent
Finally we prove the upper bound for any subsolution:

\begin{proposition}\label{uba}
Let $R>0$ be small enough depending on $N$, $p$, $s$, and $\Omega$, and let  $u\in\widetilde{W}^{s,p}(D_R)$, $M,K,H>0$ satisfy
\beq\label{uba1}
\begin{cases}
\fpl u \le \min\{K,H\} & \text{in $D_R$} \\
u \le M\ds & \text{in $\R^N$.}
\end{cases}
\eeq
Then, there exist $C>1>\sigma>0$, depending on $N$, $p$, $s$, and $\Omega$, s.t.\
\beq\label{uba2}
\inf_{D_{R/4}}\Big(M-\frac{u}{\ds}\Big) \ge \sigma L(u,M,R)-C(KR^s)^\frac{1}{p-1}-C(M+HM^{2-p})R^s.
\eeq
\end{proposition}
\begin{proof}
The argument goes exactly as in Proposition \ref{lba}. Let $\gamma_2>1$ be as in Lemma \ref{ubl}:
\begin{itemize}[leftmargin=1cm]
\item[$(a)$] If $L(u,M,R)\ge M\gamma_2$, then we use Lemma \ref{ubl}: by \eqref{uba1} $u$ satisfies \eqref{ubn1}, so \eqref{ubl1} holds.
\item[$(b)$] If instead $L(u,M,R)<M\gamma_2$, then we use Lemma \ref{ubs}: by \eqref{uba1} $u$ satisfies \eqref{ubs1} as well, so \eqref{ubs2} holds.
\end{itemize}
Taking the smallest $\sigma$ and the biggest $C$, in either case we deduce \eqref{uba2} and conclude.
\end{proof}

\section{Oscillation estimate and conclusion}\label{sec5}

\noindent
The core of our result is the following oscillation estimate for the quotient $u/\ds$, where $u$ is a function s.t.\ $\fpl u$ is bounded (in a weak sense) in $\Omega$. The next result is analogous to \cite[Theorem 5.1]{IMS1}, but the proof here follows a different path due to the singular nature of the operator for $p\in(1,2)$ (see the discussion in Section \ref{sec1}):

\begin{proposition}\label{osc}
Let $u\in W^{s,p}_0(\Omega)$, $K>0$ satisfy
\beq\label{osc1}
\begin{cases}
|\fpl u| \le K & \text{in $\Omega$} \\
u = 0 & \text{in $\Omega^c$.}
\end{cases}
\eeq
Then, there exist $\alpha\in(0,s)$, $C,R_0>0$ depending on $N$, $p$, $s$, and $\Omega$, s.t.\ for all $x_0\in\partial\Omega$ and all $r\in(0,R_0)$
\[\underset{D_r(x_0)}{\rm osc}\,\frac{u}{\ds} \le CK^\frac{1}{p-1}r^\alpha.\]
\end{proposition}

\begin{proof}
Up to a translation we may assume $x_0=0$. Also, since $\fpl$ is $(p-1)$-homogeneous, we assume $K=1$. Set for all $x\in\R^N$
\[v(x) = \begin{cases}
\displaystyle\frac{u(x)}{\ds(x)} & \text{if $x\in\Omega$} \\
0 & \text{if $x\in\Omega^c$.}
\end{cases}\]
We fix a radius $R_0=R_0(N,p,s,\Omega)$ satisfying
\[0 < R_0 < \min\Big\{\frac{\rho_\Omega}{8},\,\rho'_\Omega,\,1\Big\}\]
(with $\rho'_\Omega>0$ defined by Proposition \ref{bar}). Then, for all $n\in\N$ we set $R_n=R_0/8^n$ and as in Section \ref{sec2} 
\[D_n = D_{R_n}, \qquad \tilde B_n = \tilde B_{R_n/2},\]
so that $\tilde B_n\subset D_n$. For future use, we set for all $m\in\R$, $R>0$
\[E_+(u,m,R) = 2\sup_{x\in D_{R/2}}\int_{\{u<m\ds\}}\frac{(m\ds(y)-u(x))^{p-1}-(u(y)-u(x))^{p-1}}{|x-y|^{N+ps}}\,dy,\]
\[E_-(u,m,R) = 2\sup_{x\in D_{R/2}}\int_{\{u>m\ds\}}\frac{(u(y)-u(x))^{p-1}-(m\ds(y)-u(x))^{p-1}}{|x-y|^{N+ps}}\,dy,\]
and we note the symmetry relation
\beq\label{osc2}
E_+(u,m,R) = E_-(-u,-m,R).
\eeq
We claim that there exist $\alpha\in(0,s)$, $R_0>0$ (obeying the limitations above), and $\mu>1$, depending on $N$, $p$, $s$, and $\Omega$, and two sequences $(m_n)$, $(M_n)$ in $\R\setminus\{0\}$ (possibly depending also on $u$) s.t.\ $(m_n)$ is nondecreasing, $(M_n)$ is nonincreasing, and for all $n\in\N$
\beq\label{osc3}
m_n \le \inf_{D_n}\,v \le \sup_{D_n}\,v \le M_n, \qquad M_n-m_n = \mu R_n^\alpha.
\eeq
We argue by (strong) induction on $n$. First, let $n=0$. By \cite[Theorem 4.4]{IMS} there exists $C_0=C_0(N,p,s,\Omega)>1$ s.t.\ for all $x\in\Omega$
\[|v(x)| \le C_0.\]
For  $\alpha\in(0,s)$, $R_0>0$ to be better determined later, choose
\[\mu = \frac{2C_0}{R_0^\alpha}>1,\]
and set $M_0=-m_0=\mu R_0^\alpha/2$. Then clearly we have \eqref{osc3} at step $0$.
\vskip2pt
\noindent
Now fix $n\ge 0$, and assume that sequences $(m_k)$, $(M_k)$ are defined for $k=0,\ldots n$ and satisfy \eqref{osc3}. In particular, for all $k\in\{0,\ldots n\}$ we have
\begin{equation}\label{osc4}
|M_k|+|m_k|\le |M_0|+|m_0|\le \mu R_0^\alpha=2C_0.
\end{equation}
By \eqref{osc3} (at step $n$) we have $u\ge m_n\ds$ in $D_n$. So, by Proposition \ref{spp} we have $(u\vee m_n\ds)\in\widetilde{W}^{s,p}(D_{R_n/2})$. Using also \eqref{osc1} (recall that $K=1$), we see that for all $x\in D_{R_n/2}$
\begin{align*}
\fpl(u\vee m_n\ds)(x) &= \fpl u(x)+2\hspace{-1pt}\int_{\{u<m_n\ds\}}\hspace{-13pt}\frac{(u(x)-m_n\ds(y))^{p-1}\hspace{-1pt}-(u(x)-u(y))^{p-1}}{|x-y|^{N+ps}}dy \\
&\ge -1-E_+(u,m_n,R_n).
\end{align*}
Similarly we get $(u\wedge M_n\ds)\in\widetilde{W}^{s,p}(D_{R_n/2})$ and for all $x\in D_{R_n/2}$
\[\fpl(u\wedge M_n\ds)(x) \le 1+E_-(u,M_n,R_n).\]
In the next lines, we will provide some estimates for the quantities $E_+(u,m_n,R_n)$ and $E_-(u,M_n,R_n)$, assuming when necessary some restrictions on $m_n$, $M_n$, respectively. By the symmetry relation \eqref{osc2}, any estimate on $E_+$ will reflect on an analogous estimate on $E_-$. Preliminarily, as in \cite[Theorem 5.1]{IMS1}, for all $q>0$, $\alpha\in(0,s/q)$ we set
\[S_q(\alpha) = \sum_{j=0}^\infty\frac{(8^{\alpha j}-1)^q}{8^{sj}},\]
and note that the series above converges uniformly with respect to $\alpha$ and
\beq\label{osc5}
\lim_{\alpha\to 0^+}S_q(\alpha) = 0.
\eeq
First we focus on $E_+(u,m_n, R_n)$ and prove for such quantity an estimate involving $S_{p-1}(\alpha)$. Fix $x\in D_{R_n/2}$, $y\in\R^N$ s.t.\ $u(y)<m_n\ds(y)$, hence in particular $y\in D_n^c$ (by \eqref{osc3} at step $n$). Elementary geometric observations lead to
\[|x-y| \ge \frac{|y|}{2}, \qquad |y| \ge {\rm d}_\Omega(y).\]
By inequality \eqref{ei5} with $a=u(x)$, $b=u(y)$, and $c=m_n\ds(y)$ (note that $b\le c$) we have
\[(m_n\ds(y)-u(x))^{p-1}-(u(y)-u(x))^{p-1} \le 2^{2-p}(m_n\ds(y)-u(y))^{p-1}.\]
So, for all $x\in D_{R_n/2}$ we have
\begin{align*}
\int_{\{u<m_n\ds\}}&\frac{(m_n\ds(y)-u(x))^{p-1}-(u(y)-u(x))^{p-1}}{|x-y|^{N+ps}}\,dy \\
& \ \le C\int_{\{u<m_n\ds\}}\frac{(m_n\ds(y)-u(y))^{p-1}}{|y|^{N+ps}}\,dy \\
& \ \le C\int_{\Omega\cap D_n^c}\frac{(m_n-v(y))_+^{p-1}}{|y|^{N+s}}\,dy = C{\rm tail}_{p-1}^{p-1}((m_n-v)_+,R_n),
\end{align*}
where the tail is defined as in \eqref{tail} with $q=p-1$, and $C=C(N,p,s,\Omega)>0$. Note that the last quantity does not depend on $x$. We split the integral defining 
\[A_1 = \Omega\setminus D_1, \qquad A_k = D_{k-1}\setminus D_k \quad (k=2,\ldots n),\]
and noting that for some $C=C(N, p, s)>0$
\[\int_{A_k}\frac{1}{|y|^{N+s}}\,dy \le \int_{D_k^c}\frac{1}{|y|^{N+s}}\,dy \le \frac{C}{R_{k-1}^s}.\]
Then, for all $k\in\{1,\ldots n\}$, $y\in A_k$, we apply \eqref{osc3} and monotonicity of the sequences $(m_n)$, $(M_n)$ to get
\begin{align*}
m_n-v(y) &\le m_n-m_{k-1} \\
&\le (m_n-M_n)+(M_{k-1}-m_{k-1}) = \mu(R_{k-1}^\alpha-R_n^\alpha).
\end{align*}
Also, recall that $S_{p-1}(\alpha)$ converges due to $\alpha<s/(p-1)$. Therefore, splitting the integral, we have the following estimate for the tail:
\begin{align*}
{\rm tail}_{p-1}^{p-1}\big((m_n-v)_+,R_n\big) &=\sum_{k=1}^n \int_{A_k}\frac{(m_n-v(y))_+^{p-1}}{|y|^{N+s}}\,dy \\
&\le  \sum_{k=1}^n\int_{A_k}\frac{\mu^{p-1}(R_{k-1}^\alpha-R_n^\alpha)^{p-1}}{|y|^{N+s}}\,dy \\
&\le C\mu^{p-1}\sum_{k=1}^n\frac{(R_{k-1}^\alpha-R_n^\alpha)^{p-1}}{R_{k-1}^s} \\
&=  C\mu^{p-1}S_{p-1}(\alpha)R_n^{(p-1)\alpha-s}.
\end{align*}
Going back to $E_+(u,m_n,R_n)$, we have found $C=C(N,p,s,\Omega)>0$ s.t.\
\beq\label{osc6}
E_+(u,m_n, R_n) \le  C\mu^{p-1}S_{p-1}(\alpha)R_n^{(p-1)\alpha-s}.
\eeq
The estimate in \eqref{osc6}, though fairly general, is not fully satisfactory, as it involves an exponent of the radius which is less than $\alpha-s$. So we need another estimate of $E_+(u,m_n,R_n)$ involving $S_1(\alpha)$, which we first prove under the assumption $m_n>0$. For any $x\in D_{R_n/2}$ we split the integral appearing in $E_+(u,m_n,R_n)$, defining the subdomains $A_k$ ($k=1,\ldots n$) and recalling that $\{u<m_n \ds\}\subseteq D_n^c$:
\begin{align*}
\int_{\{u<m_n\ds\}}&\frac{(m_n\ds(y)-u(x))^{p-1}-(u(y)-u(x))^{p-1}}{|x-y|^{N+ps}}\,dy \\
& \ \le \sum_{k=1}^n\int_{A_k}\frac{\big[(m_n\ds(y)-u(x))^{p-1}-(u(y)-u(x))^{p-1}\big]_+}{|x-y|^{N+ps}}\,dy =\sum_{k=1}^n I_k(x).
\end{align*}
For all $x\in D_{R_n/2}$, $y\in A_k$ we have 
\[|x-y| \ge R_{k}-\frac{R_n}{2}\ge \frac{R_k}{2}=\frac{R_{k-1}}{16}\]
and by \eqref{osc3} at step $k-1$ it holds $u(y) \ge m_{k-1}\ds(y)$.
So we find $C=C(N, p, s)$ s.t.\
\[I_k(x) \le \frac{C}{R_{k-1}^{N+ps}}\int_{A_k}\big[(m_n\ds(y)-u(x))^{p-1}-(m_{k-1}\ds(y)-u(x))^{p-1}\big]\,dy\]
We recall that $m_n>0$, while we have no sign information on $m_{k-1}$. So we distinguish two cases:
\begin{itemize}[leftmargin=1cm]
\item[$(a)$] If $m_{k-1}\le m_n/2$ (including $m_{k-1}\le 0$), then we have
\[m_n-m_{k-1} \ge \frac{m_n}{2}.\]
So for all $y\in A_k$ we use \eqref{ei5} with $a=u(x)$, $b=m_{k-1}\ds(y)$, and $c=m_n\ds(y)$ ($b\le c$) to find
\begin{align*}
(m_n\ds(y)-u(x))^{p-1}-(m_{k-1}\ds(y)-u(x))^{p-1} &\le 2^{2-p}(m_n-m_{k-1})^{p-1}{\rm d}_\Omega^{(p-1)s}(y) \\
&\le 4^{2-p}\, \frac{m_n-m_{k-1}}{m_n^{2-p}}\, {\rm d}_\Omega^{(p-1)s}(y).
\end{align*}
Plugging such inequality into the previous estimate of $I_k(x)$ we get, for $k>1$,
\begin{align*}
I_k(x) &\le \frac{ C}{R_{k-1}^{N+ps}}\int_{A_k}\frac{m_n-m_{k-1}}{m_n^{2-p}}{\rm d}_\Omega^{(p-1)s}(y)\, dy \\
&\le \frac{C}{R_{k-1}^s}\,\frac{m_n-m_{k-1}}{m_n^{2-p}}
\end{align*}
with $C=C(N,p,s,\Omega)>0$. Besides, for $k=1$ we have
\begin{align*}
I_1(x) &\le \frac{C}{R_0^{N+ps}}\int_\Omega\frac{m_n-m_0}{m_n^{2-p}}{\rm d}_\Omega^{(p-1)s}(y)\,dy \\
&\le \frac{C}{R_0^{N+(p-1)s}R_0^s}\frac{m_n-m_0}{m_n^{2-p}},
\end{align*}
again with $C=C(N,p,s,\Omega)>0$. Since $R_0<1$, in both cases we find $C=C(N,p,s,\Omega)>0$ s.t.\
\[I_k(x) \le \frac{C}{R_0^{N+(p-1)s}R_{k-1}^s}\,\frac{m_n-m_{k-1}}{m_n^{2-p}}.\]
\item[$(b)$] If $m_n/2<m_{k-1}\le m_n$, then necessarily $k>1$ (since $m_n>0>m_0$). Set for $x\in D_{R_n/2}$ and $y\in A_k$
\[h_x(y) = (m_n\ds(y)-u(x))^{p-1}-(m_{k-1}\ds(y)-u(x))^{p-1} \ge 0.\]
Recall that $|\nabla{\rm d}_\Omega|=1$ a.e.\ in $A_k$. Therefore, by the coarea formula we have for all $x\in D_{R_n/2}$
\beq\label{osc7}
\int_{A_k}h_x(y)\,dy =\int_{A_k}h_x(y)|\nabla{\rm d}_\Omega(y)|\,dy = \int_{\R}\int_{A_k\cap\{{\rm d}_\Omega=\xi\}}h_x(y)\,d{\mathcal H}^{N-1}\,d\xi,
\eeq
where $\mathcal{H}^{N-1}$ is $(N-1)$-dimensional Hausdorff measure on  $A_k\cap\{{\rm d}_\Omega=\xi\}$. By Lemma \ref{lev}, there exists $C=C(\Omega)>0$ s.t.\ for all $k>1$ and all $\xi>0$
\[{\mathcal H}^{N-1}\big(A_k\cap\{{\rm d}_\Omega=\xi\}\big)\le CR_{k-1}^{N-1}.\]
Using the above measure-theoretic bound in \eqref{osc7}, we find $C=C(N,p,s,\Omega)>0$ s.t.\ for all $x\in D_{R_n/2}$
\[\int_{A_k}h_x(y)\, dy\le CR_{k-1}^{N-1}\int_0^{R_{k-1}}\big[(m_n\xi^s-u(x))^{p-1}-(m_{k-1}\xi^s-u(x))^{p-1}\big]\,d\xi.\]
Set for all $m>0$, $t\in\R$
\[\psi_k(m,t) = \int_0^{R_{k-1}}(m\xi^s-t)^{p-1}\,d\xi.\]
Then we have $\psi_k(\cdot,t)\in C^1(0,\infty)$ with derivative
\[\frac{\partial\psi_k}{\partial m}(m,t) = (p-1)m^{p-2}\int_0^{R_{k-1}}\Big|\xi^s-\frac{t}{m}\Big|^{p-2}\xi^s\,d\xi.\]
We estimate the last integral by using the change of variable $\xi=R_{k-1}\eta$:
\begin{align*}
\int_0^{R_{k-1}}\Big|\xi^s-\frac{t}{m}\Big|^{p-2}\xi^s\,d\xi &= R_{k-1}^{1+s}\int_0^1\Big|R_{k-1}^s\eta^s-\frac{t}{m}\Big|^{p-2}\eta^s\,d\eta \\
&= R_{k-1}^{(p-1)s+1}\int_0^1\Big|\eta^s-\frac{t}{mR_{k-1}^s}\Big|^{p-2}\eta^s\,d\eta.
\end{align*}
By the further change of variable $\theta=\eta^s$ and \eqref{ei5}, we have for all $\tau\in\R$
\begin{align*}
\int_0^1|\eta^s-\tau|^{p-2}\eta^s\,d\eta &\le \int_0^1|\eta^s-\tau|^{p-2}\eta^{s-1}\,d\eta \\
&= \int_0^1|\theta-\tau|^{p-2}\frac{d\theta}{s} \\
&= \frac{(1-\tau)^{p-1}-\tau^{p-1}}{(p-1)s} \le \frac{2^{2-p}}{(p-1)s}.
\end{align*}
In conclusion, there exists $C=C(p,s)>0$ s.t.\ for all $m>0$, $t\in\R$, $k>1$
\beq\label{osc8}
\frac{\partial\psi_k}{\partial m}(m,t) \le CR_{k-1}^{(p-1)s+1}.
\eeq
We now go back to $I_k(x)$. We apply Lagrange's theorem to $\psi_k(\cdot,u(x))$ in the interval $[m_{k-1},m_n]\subset(0,\infty)$, along with \eqref{osc8} and the previous relations, to get
\begin{align*}
I_k (x)&\le \frac{C}{R_{k-1}^{N+ps}}\int_{A_k}h_x(y)\,dy \\
&\le \frac{C R_{k-1}^{N-1}}{R_{k-1}^{N+ps}}\big[\psi_k(m_n,u(x))-\psi_k(m_{k-1},u(x))\big] \\
&\le \frac{C}{ R_{k-1}^{1+ps}}\max_{m_{k-1}\le m\le m_n}\frac{\partial\psi_k}{\partial m}(m,u(x))(m_n-m_{k-1}) \\
&\le \frac{C}{R_{k-1}^s}\,\frac{m_n-m_{k-1}}{m_{k-1}^{2-p}}.
\end{align*}
Recalling that $m_{k-1}\ge m_n/2$, we conclude that there exists $C=C(N,p,s,\Omega)>0$ s.t.\
\[I_k(x) \le \frac{C}{R_{k-1}^s}\,\frac{m_n-m_{k-1}}{m_{n}^{2-p}}.\]
\end{itemize}
In both cases $(a)$, $(b)$, recalling that $R_0<1$, we have reached the same estimate of $I_k(x)$, which is independent from $x\in D_{R_n/2}$ and of the form
\[I_k(x) \le \frac{C}{R_0^{N+(p-1)s}R_{k-1}^s}\,\frac{m_n-m_{k-1}}{m_{n}^{2-p}},\]
for  $C=C(N,p,s,\Omega)>0$ independent of $R_0$. Thus, by \eqref{osc3} at steps $k=0,\ldots n$ we have
\begin{align*}
E_+(u,m_n, R_n) &\le  \sum_{k=1}^n\frac{C}{R_0^{N+(p-1)s}R_{k-1}^s}\,\frac{m_n-m_{k-1}}{m_n^{2-p}} \\
&\le  \frac{C\mu}{R_0^{N+(p-1)s}m_n^{2-p}}\sum_{k=1}^n\frac{R_{k-1}^\alpha-R_n^\alpha}{R_{k-1}^s}.
\end{align*}
Recalling the definition of $S_1(\alpha)$, whenever $m_n>0$ we have the following alternative estimate involving a constant $C=C(N,p,s,\Omega)>0$:
\beq\label{osc9}
E_+(u,m_n, R_n) \le \frac{C\mu}{m_n^{2-p}}\frac{S_1(\alpha)}{R_0^{N+(p-1)s}}R_n^{\alpha-s}.
\eeq
Next we focus on the quantity $E_-(u,M_n, R_n)$. Recalling the symmetry relation \eqref{osc2}, we have
\[E_-(u,M_n, R_n) = E_+(-u,-M_n, R_n).\]
Note that the function $-u\in W^{s,p}_0(\Omega)$ satisfies \eqref{osc1}, and by \eqref{osc3} at step $n$ we have for all $x\in D_n$
\[-u(x) \ge -M_n\ds(x).\]
So, arguing exactly as in \eqref{osc6} we find $C=C(N,p,s,\Omega)>0$ s.t.\
\beq\label{osc10}
E_-(u,M_n, R_n) \le C\mu^{p-1}S_{p-1}(\alpha)R_n^{(p-1)\alpha-s}.
\eeq
The argument for a \eqref{osc9}-type estimate for $E_-(u,M_n, R_n)$ is slightly different. We assume $M_n>0$ and for all $x\in D_{R_n/2}$ we define $A_k$ ($k=1,\ldots n$) as above and split the corresponding integral as follows:
\begin{align*}
\int_{\{u>M_n\ds\}}&\frac{(u(y)-u(x))^{p-1}-(M_n\ds(y)-u(x))^{p-1}}{|x-y|^{N+ps}}\,dy \\
&\ \le \sum_{k=1}^n\int_{A_k}\frac{\big[(u(y)-u(x))^{p-1}-(M_n\ds(y)-u(x))^{p-1}\big]_+}{|x-y|^{N+ps}}\,dy \\
& \ \le \sum_{k=1}^n\frac{C}{R_{k-1}^{N+ps}}\int_{A_k}\big[(M_{k-1}\ds(y)-u(x))^{p-1}-(M_n\ds(y)-u(x))^{p-1}\big]\,dy.
\end{align*}
For $k>1$ we can argue as in case $(b)$ above to get, thanks to $M_{k-1}\ge M_n$ and \eqref{osc8}, that for all $x\in D_{R_n/2}$
\begin{align*}
\frac{C}{R_{k-1}^{N+ps}}&\int_{A_k}\big[(M_{k-1}\ds(y)-u(x))^{p-1}-(M_n\ds(y)-u(x))^{p-1}\big]\,dy \\
& \ \le \frac{C}{R_{k-1}^{1+ps}}\max_{M_n\le m\le M_{k-1}}\frac{\partial\psi_k}{\partial m}(m, u(x)) (M_{k-1}-M_n) \\
&\ \le \frac{C}{R_{k-1}^s}\frac{M_{k-1}-M_n}{M_n^{2-p}}.
\end{align*}
Now turn to $k=1$. First, by \cite[Corollary 2]{K1} we have the following uniform bound for all $\xi\ge 0$ and some $C=C(N,\Omega)>0$:
\[{\mathcal H}^{N-1}(\{{\rm d}_\Omega=\xi\}) \le C.\]
We apply the coarea formula as in case $(b)$ above to get
\[\int_{A_1}\big[(M_{0}\ds(y)-u(x))^{p-1}-(M_n\ds(y)-u(x))^{p-1}\big]\,dy \le C\big[\psi_1(M_0, u(x))-\psi_1(M_n, u(x))\big],\]
where $C=C(N,p,s,\Omega)>0$ and for all $m>0$, $t\in\R$
\[\psi_1(m,t) = \int_0^{{\rm diam}\, (\Omega)}(m\xi^s-t)^{p-1}\,d\xi.\]
Arguing as in \eqref{osc8}, we find $C=C(N,p,s,\Omega)>0$ s.t.\ for all $m>0$, $t\in\R$
\[\frac{\partial \psi_1}{\partial m}(m,t)\le \frac{C}{m^{2-p}}.\]
So, by Lagrange's theorem we get
\begin{align*}
\int_{A_1}\hspace{-1pt}(M_0\ds(y)-u(x))^{p-1}\hspace{-1pt}-(M_n\ds(y)-u(x))^{p-1}\,dy &\le C\hspace{-3pt}\max_{M_n\le m\le M_0}\hspace{-2pt}\frac{\partial\psi_1}{\partial m}(m,u(x))(M_0-M_n) \\
&\le C\frac{M_0-M_n}{M_n^{2-p}}.
\end{align*}
Since $R_0<1$, we have
\[\int_{\{u>M_n\ds\}}\hspace{-6pt}\frac{(u(y)-u(x))^{p-1}-(M_n\ds(y)-u(x))^{p-1}}{|x-y|^{N+ps}}\,dy \le 
\frac{C}{R_0^{N+(p-1)s}}\sum_{k=1}^n\frac{1}{R_{k-1}^s}\frac{M_{k-1}-M_n}{M_n^{2-p}}.\]
Thus, whenever $M_n>0$ we have the following estimate, with $C=C(N,p,s,\Omega)>0$:
\beq\label{osc11}
E_-(u,M_n, R_n) \le  \frac{C\mu}{M_n^{2-p}}\frac{S_1(\alpha)}{R_0^{N+(p-1)s}}R_n^{\alpha-s}.
\eeq
All the estimates \eqref{osc6}, \eqref{osc9}, \eqref{osc10}, and \eqref{osc11} hold if $0<m_n<M_n$. We briefly hint at the remaining cases (recalling that by \eqref{osc3} at step $n$ we have $m_n,M_n\neq 0$ and $m_n<M_n$):
\begin{itemize}[leftmargin=1cm]
\item[$(A)$] If $m_n<M_n<0$, then conversely $0<-M_n<-m_n$, so we turn to the function $-u$ which solves \eqref{osc1} and satisfies for all $x\in D_n$
\[-M_n\ds(x) \le -u(x) \le -m_n\ds(x).\]
Arguing as above and using \eqref{osc2}, we find estimates of the type \eqref{osc6} - \eqref{osc11} for the quantities
\[E_+(u,m_n, R_n) = E_-(-u,-m_n, R_n), \qquad E_-(u,M_n, R_n) = E_+(-u,-M_n, R_n).\]
\item[$(B)$] If $m_n<0<M_n$, then we can derive \eqref{osc10}, \eqref{osc11} as above. Besides, passing to $-u$ and noting that $-m_n>0$, we find estimates like \eqref{osc6}, \eqref{osc9} for
\[E_+(u,m_n, R_n) = E_-(-u,-m_n, R_n).\]
\end{itemize}
Summarizing, in any case we find $C=C(N,p,s,\Omega)>0$ s.t.\ 
\[E_+(u,m_n, R_n) \le C\min\Big\{\mu^{p-1} S_{p-1}(\alpha)R_n^{(p-1)\alpha-s},\, \frac{\mu}{|m_n|^{2-p}}\frac{S_1(\alpha)}{R_0^{N+(p-1)s}}R_n^{\alpha-s}\Big\},\]
\[E_-(u,M_n, R_n) \le C\min\Big\{\mu^{p-1}S_{p-1}(\alpha)R_n^{(p-1)\alpha-s},\,\frac{\mu}{|M_n|^{2-p}}\frac{S_1(\alpha)}{R_0^{N+(p-1)s}}R_n^{\alpha-s}\Big\}.\]
The next step consists in applying the lower and upper bounds proved in Sections \ref{sec3}, \ref{sec4} to the functions $u\vee m_n\ds$, $u\wedge M_n\ds$, respectively. To this end, note that  both Proposition \ref{lba} and Proposition \ref{uba}, while {\em separately} proved in the case $m,M>0$, actually hold true {\em together} for arbitrary $m,M\ne 0$: indeed, if $(u,m)$ fulfill the assumptions of Proposition \ref{lba} for some $m<0$, then Proposition \ref{uba} applies to $(-u,-m)$, and \eqref{uba2} is then equivalent to \eqref{lba2} with $|m|$ on the right hand side; similarly, if $(u,M)$ satisifes the assumptions of Proposition \ref{uba} for some $M<0$, then Proposition \ref{lba} applies to $(-u,-M)$, giving \eqref{uba2} with $|M|$ on the right hand side.
\vskip2pt
\noindent
Set then, for $C=C(N, p, s, \Omega)>1$ given in the previous bounds for $E_\pm$, 
\[K_n = 1+C\mu^{p-1}S_{p-1}(\alpha)R_n^{(p-1)\alpha-s},\]
\[h_n = 1+\frac{C\mu}{|m_n|^{2-p}}\frac{S_1(\alpha)}{R_0^{N+(p-1)s}}R_n^{\alpha-s},\]
\[H_n = 1+\frac{C\mu}{|M_n|^{2-p}}\frac{S_1(\alpha)}{R_0^{N+(p-1)s}}R_n^{\alpha-s}.\]
By the previous bound on $\fpl(u\vee m_n\ds)$ and the estimates on $E_+(u,m_n,R_n)$ in \eqref{osc6} and \eqref{osc9}, we have the following \eqref{lba1}-type inequality:
\[\begin{cases}
\fpl(u\vee m_n\ds) \ge -\min\{K_n,h_n\} & \text{in $D_{R_n/2}$} \\
u\vee m_n\ds \ge m_n\ds & \text{in $\R^N$.}
\end{cases}\]
We apply Proposition \ref{lba} with $R=R_n/2$, $m=m_n$, $K=K_n$, $H=h_n$. Recalling that $R_n/8=R_{n+1}$ and that $u\ge m_n\ds$ in $D_n$ by \eqref{osc3} at step $n$, with slightly rephrased constants, we find $ C>1>\sigma>0$, depending on $N$, $p$, $s$, and $\Omega$ (but not on $R_0$), s.t.\
\[\inf_{D_{n+1}}\big(v-m_n\big) \ge \sigma L\Big(u,m_n,\frac{R_n}{2}\Big)-C(K_nR_n^s)^\frac{1}{p-1}- C(|m_n|+h_n|m_n|^{2-p})R_n^s.\]
Similarly, by \eqref{osc10} and \eqref{osc11} we have the \eqref{uba1}-type inequality
\[\begin{cases}
\fpl(u\wedge M_n\ds) \le \min\{K_n,H_n\} & \text{in $D_{R_n/2}$} \\
u\wedge m_n\ds \le M_n\ds & \text{in $\R^N$.}
\end{cases}\]
By Proposition \ref{uba}, for $C>1$ even bigger and $\sigma\in(0,1)$ even smaller if necessary, all depending on $N, p, s, \Omega$ (but not on $R_0$), we have
\[\inf_{D_{n+1}}\big(M_n-v\big) \ge \sigma L\Big(u,M_n,\frac{R_n}{2}\Big)-C(K_nR_n^s)^\frac{1}{p-1}- C(|M_n|+H_n|M_n|^{2-p})R_n^s.\]
Comparing the definitions of $K_n$, $h_n$, and $H_n$ with the lower-upper bounds above, we realize that the latter show a certain degree of homogeneity, which we are now going to exploit in the final steps of the proof. We proceed to estimating the oscillation of $v=u/\ds$ in $D_{n+1}$ (recalling the definition of the excess in \eqref{lum}):
\begin{align*}
\underset{D_{n+1}}{\rm osc}\,v &= \sup_{D_{n+1}}v-\inf_{D_{n+1}}v\\
&\le (M_n-m_n)-\inf_{D_{n+1}}\big(v-m_n\big)-\inf_{D_{n+1}}\big(M_n-v\big) \\
&\le (M_n-m_n)-\sigma\Big[L\Big(u,m_n,\frac{R_n}{2}\Big)+L\Big(u,M_n,\frac{R_n}{2}\Big)\Big] \\
& \quad + C(K_nR_n^s)^\frac{1}{p-1}+ C\big[|m_n|+h_n|m_n|^{2-p}+|M_n|+H_n|M_n|^{2-p}\big]R_n^s \\
&\le (M_n-m_n)-\sigma\Big[\Big(\dashint_{\tilde B_n}(v(x)-m_n)^{p-1}\,dx\Big)^\frac{1}{p-1}+\Big(\dashint_{\tilde B_n}(M_n-v(x))^{p-1}\,dx\Big)^\frac{1}{p-1}\Big] \\
& \quad +C\mu S_{p-1}^\frac{1}{p-1}(\alpha)R_n^{\alpha}   +C \frac{\mu S_1(\alpha)}{R_0^{N+(p-1)s}}R_n^{\alpha} + C\mu R_0^\alpha R_n^s,
\end{align*}
 where in the last step we used \eqref{osc4}. By subadditivity of $t\mapsto t^{p-1}$ on $[0,\infty)$, for all $x\in\tilde B_n$ it holds
\[(M_n-m_n)^{p-1} \le (M_n-v(x))^{p-1}+(v(x)-m_n)^{p-1}.\]
Using inequality \eqref{ei4}, we thus have
\begin{align*}
\Big(\dashint_{\tilde B_n}(v(x)-m_n)^{p-1}\,dx\Big)^\frac{1}{p-1}&+\Big(\dashint_{\tilde B_n}(M_n-v(x))^{p-1}\,dx\Big)^\frac{1}{p-1} \\
& \ge 2^\frac{p-2}{p-1}\Big[\dashint_{\tilde B_n}(v(x)-m_n)^{p-1}\,dx+\dashint_{\tilde B_n}(M_n-v(x))^{p-1}\,dx\Big]^\frac{1}{p-1} \\
& \ge 2^\frac{p-2}{p-1}\Big[\dashint_{\tilde B_n}(M_n-m_n)^{p-1}\,dx\Big]^\frac{1}{p-1} = 2^\frac{p-2}{p-1}\big(M_n-m_n\big).
\end{align*}
Plugging this inequality into the previous oscillation estimate, using \eqref{osc3}, and recalling that $R_n\le R_0$ and $\mu>1$, we have
\begin{align*}
\underset{D_{n+1}}{\rm osc}\,v &\le \Big(1-2^\frac{p-2}{p-1}\sigma\Big)(M_n-m_n)+C\mu\Big[S_{p-1}^\frac{1}{p-1}(\alpha)+\frac{S_1(\alpha)}{R_0^{N+(p-1)s}}\Big]R_n^\alpha+ C\mu R_0^\alpha R_n^s \\
&\le \Big[1-2^\frac{p-2}{p-1}\sigma+CS_{p-1}^\frac{1}{p-1}(\alpha)+C\frac{S_1(\alpha)}{R_0^{N+(p-1)s}}\Big]8^\alpha\mu R_{n+1}^\alpha+ C\mu R_0^{s}  R_{n+1}^\alpha,
\end{align*}
where in the last passage we used the inequality
\[R_0^\alpha R_n^s \le 8^\alpha R_0^s R_{n+1}^\alpha.\]
So far, $\alpha\in(0,s)$ and $R_0>0$ (obeying the initial bounds) are not subject to any further restriction. We now choose $R_0=R_0(N, p, s, \Omega)>0$ small enough s.t.\ 
\[CR_0^{s} \le 2^{\frac{p-2}{p-1}-1}\sigma,\]
and correspondingly, thanks to \eqref{osc5}, choose $\alpha=\alpha(N, p, s, \Omega)\in (0,s)$ so small that 
\[\Big[1-2^\frac{p-2}{p-1}\sigma+C S_{p-1}^\frac{1}{p-1}(\alpha)+C\frac{S_1(\alpha)}{R_0^{N+(p-1)s}}\Big]8^\alpha < 1-2^{\frac{p-2}{p-1}-1}\sigma.\]
By virtue of such relations we have
\[\underset{D_{n+1}}{\rm osc}\,v \le \mu R_{n+1}^\alpha.\]
Therefore, we may fix $m_{n+1},M_{n+1}\in[m_n,M_n]$ s.t.\
\[m_{n+1} \le \inf_{D_{n+1}}v \le \sup_{D_{n+1}}v \le M_{n+1}, \qquad M_{n+1}-m_{n+1} = \mu R_{n+1}^\alpha,\]
thus proving \eqref{osc3} at step $n+1$.
\vskip2pt
\noindent
Finally, let $r\in(0,R_0)$. Then, there exists $n\in\N$ s.t.\ $R_{n+1} \le r < R_n$. By \eqref{osc3} we have
\[\underset{D_r}{\rm osc}\,v \le \underset{D_n}{\rm osc}\,v \le \mu R_n^\alpha \le \mu 8^\alpha r^\alpha.\]
Thus, the conclusion holds with $\alpha\in(0,s)$ and $C=\mu 8^\alpha>1$, both depending on $N$, $p$, $s$, and $\Omega$.
\end{proof}

\noindent
We can now prove our main result. The argument is in fact identical to that of \cite[Theorem 1.1]{IMS1}, but we include it here for completeness:
\vskip4pt
\noindent
{\em Proof of Theorem \ref{main}.} As said in Section \ref{sec1}, case $p\ge 2$ is just \cite[Theorem 1.1]{IMS1}, so we assume $p\in(1,2)$. Let $f\in L^\infty(\Omega)$, $u\in W^{s,p}(\Omega)$ be a solution of \eqref{dir}. Set $K=\|f\|_{L^\infty(\Omega)}$, so $u$ satisfies \eqref{osc1}. By homogeneity of $\fpl$, we may assume $K=1$. Set as before for all $x\in\R^N$
\[v(x) = \begin{cases}
\displaystyle\frac{u(x)}{\ds(x)} & \text{if $x\in\Omega$} \\
0 & \text{if $x\in\Omega^c$.}
\end{cases}\]
For $\alpha\in (0, s)$ being given in Proposition \ref{osc}, we aim at applying Lemma \ref{ros} to $v$ with $\gamma=\alpha$. 
As already recalled (see \eqref{opt1}), we have $v\in L^\infty(\Omega)$ and there is $C=C(N,p,s,\Omega)>0$ s.t.\
\[\|v\|_{L^\infty(\Omega)}\le C,\]
hence $v$ satisfies hypothesis \ref{ros1} of Lemma \ref{ros}. In order to check \ref{ros2}, let $x_1\in\Omega$ be s.t.\ ${\rm d}_\Omega(x_1)=4R$, and arguing as in Theorem \ref{opt} (with $\gamma=\alpha$) we get
\[[u]_{C^\alpha(B_{R/8}(x_1))} \le CR^{s-\alpha},\]
aith $C=C(N,p,s,\Omega)>0$. Besides, by \cite[p.\ 292]{ROS} we have
\[\Big[\frac{1}{\ds}\Big]_{C^\alpha(B_{R/8}(\bar x))} \le \frac{C}{R^{\alpha+s}}.\]
Combining the previous properties, we have for all $x,y\in B_{R/8}(\bar x)$
\begin{align*}
\frac{|v(x)-v(y)|}{|x-y|^\alpha} &\le \frac{|u(x)-u(y)|}{|x-y|^\alpha}\frac{1}{\ds(x)}+\frac{|u(y)|}{|x-y|^\alpha}\Big|\frac{1}{\ds(x)}-\frac{1}{\ds(y)}\Big| \\
&\le [u]_{C^\alpha(B_{R/8}(\bar x))}\Big\|\frac{1}{\ds}\Big\|_{L^\infty(B_{R/8}(\bar x))}+\|u\|_{L^\infty(B_{R/8})}\Big[\frac{1}{\ds}\Big]_{C^\alpha(B_{R/8}(\bar x))} \\
&\le \frac{C R^s}{R^\alpha}\Big(\frac{8}{R}\Big)^s+CR^s\frac{C}{R^{\alpha+s}} \le \frac{C}{R^{\alpha}}.
\end{align*}
So, $v$ satisfies hypothesis \ref{ros2} of Lemma \ref{ros} with $\nu=\gamma=\alpha$. Finally, fix $x_0\in\partial\Omega$, $r>0$, and let   $C,R_0>0$ be as in Proposition \ref{osc}. We distinguish two cases:
\begin{itemize}[leftmargin=1cm]
\item[$(a)$] If $r<R_0$, then by Proposition \ref{osc} we have
\[\underset{D_r(x_0)}{\rm osc}\,v \le Cr^\alpha.\]
\item[$(b)$] If $r\ge R_0$, then simply
\[\underset{D_r(x_0)}{\rm osc}\,v \le 2\|v\|_{L^\infty(\Omega)} \le \frac{C}{R_0^\alpha}r^\alpha.\]
\end{itemize}
In any case, $v$ satisfies hypothesis \ref{ros3} of Lemma \ref{ros} with $\gamma=\alpha$ and $M=M(N, p, s, \Omega)>0$. The corresponding exponent in Lemma \ref{ros} turns out to be $\alpha/2\in(0,s)$, therefore
\[[v]_{C^{\alpha/2}(\overline\Omega)} \le C,\]
for $C=C(N, p, s, \Omega)>0$, which, along with the previous bound on $\|v\|_{L^\infty(\Omega)}$, yields the conclusion, up to replacing $\alpha$ with $\alpha/2$. \qed

\appendix

\section{Some elementary inequalities}\label{appa}

\noindent
Here we recall some useful inequalities, specifically designed to deal with the singular case. In the following, $p\in(1,2)$:
\begin{itemize}[leftmargin=1cm]
\item For all $a,b\ge 0$ we have
\beq\label{ei1}
a^{p-1}-b^{p-1} \ge (a-b)^{p-1}-a^{p-1}.
\eeq
Indeed, if $a>b$ then by monotonicity of $t\mapsto t^{p-1}$ we have
\[a^{p-1} \ge (a-b)^{p-1}, \qquad -b^{p-1} \ge -a^{p-1},\]
hence \eqref{ei1}. If $a\le b$, then by subadditivity of $t\mapsto|t|^{p-1}$ we have
\[b^{p-1} \le (b-a)^{p-1}+a^{p-1},\]
hence
\[a^{p-1}-b^{p-1} \ge (a-b)^{p-1},\]
which in turn implies \eqref{ei1}.
\item For all $a,b\ge 0$, $\theta>0$ we have
\beq\label{ei2}
\frac{a}{(a+b)^{2-p}} \ge \Big(\frac{\theta}{\theta+1}\Big)^{2-p}a^{p-1}-\frac{\theta}{(\theta+1)^{2-p}}b^{p-1}.
\eeq
Indeed, if $a<\theta b$ then
\[\Big(\frac{\theta}{\theta+1}\Big)^{2-p}a^{p-1} < \frac{\theta}{(\theta+1)^{2-p}}b^{p-1},\]
so \eqref{ei2} is trivial. If $a\ge\theta b$ then
\[(a+b)^{2-p} \le \Big(1+\frac{1}{\theta}\Big)^{2-p}a^{2-p},\]
which in turn implies \eqref{ei2}.
\item For all $a\in\R$, $b\ge 0$ we have
\beq\label{ei3}
(a+b)^{p-1}-a^{p-1} \ge b^{p-1}-|a|^{p-1}.
\eeq
Indeed, if $a\ge 0$ then by monotonicity of $t\mapsto t^{p-1}$ \eqref{ei3} is trivial. If $-b\le a<0$ then by subadditivity of $t\mapsto t^{p-1}$ on $[0,\infty)$ we have
\[b^{p-1} \le (a+b)^{p-1}+(-a)^{p-1},\]
hence \eqref{ei3}. Finally, if $a<-b$ then by monotonicity again
\[(a+b)^{p-1}+(-a)^{p-1} \ge a^{p-1}+b^{p-1},\]
which is equivalent to \eqref{ei3}.
\item For all $a,b\ge 0$ we have
\beq\label{ei4}
(a+b)^\frac{1}{p-1} \le 2^\frac{2-p}{p-1}\big(a^\frac{1}{p-1}+b^\frac{1}{p-1}\big).
\eeq
Indeed, let $q=1/(p-1)>1$, then by convexity of $t\mapsto t^q$ in $[0,\infty)$ we have
\[(a+b)^q = 2^q\Big(\frac{a}{2}+\frac{b}{2}\Big)^q \le 2^{q-1}(a^q+b^q),\]
hence \eqref{ei4}.
\item For all $a,b,c\in\R$, $b\le c$ we have
\beq\label{ei5}
(c-a)^{p-1}-(b-a)^{p-1} \le 2^{2-p}(c-b)^{p-1}.
\eeq
Indeed, by the rearrangement inequality for integrals we have
\[\int_{a-c}^{a-b}|t|^{p-2}\,dt \le \int_{(b-c)/2}^{(c-b)/2}|t|^{p-2}\,dt,\]
which rephrases as
\begin{align*}
\frac{(a-b)^{p-1}}{p-1}-\frac{(a-c)^{p-1}}{p-1} &\le \frac{1}{p-1}\Big[\Big(\frac{c-b}{2}\Big)^{p-1}-\Big(\frac{b-c}{2}\Big)^{p-1}\Big] \\
&= \frac{2^{2-p}}{p-1}(c-b)^{p-1},
\end{align*}
hence \eqref{ei5}.
\end{itemize}

\section{Proof of Proposition \ref{bar}}\label{appb}

\noindent
Here we give a sketch of the proof of Proposition \ref{bar}. In fact, the original argument of \cite[Lemma 3.4]{IMS1} only works for domains with a $C^{2,1}$-smooth boundary, as it requires that the metric projection onto the boundary be $C^{1,1}$ (see \cite{LS}). We modify the argument in order to include $C^{1,1}$-smooth domains, and so we correct the gap of the original proof.
\vskip4pt
\noindent
{\em Proof of Proposition \ref{bar}.}
The idea is the following: first, we rephrase $v_\lambda$ by means of a convenient diffeomorphism and a distance function; then, we prove the desired bound on $\fpl v_\lambda$ as in \cite[Lemma 3.4]{IMS1}.
\vskip2pt
\noindent
First we introduce a signed distance from $\partial U$ by setting for all $x\in\R^N$
\[{\rm d}(x) = \begin{cases}
{\rm d}_U(x) & \text{if $x\in U$} \\
-{\rm d}_{U^c}(x) & \text{if $x\in U^c$.}
\end{cases}\]
By the regularity of $\partial U$, ${\rm d}$ is a $C^{1,1}$-function in a convenient tubular neighborhood of $\partial U$, satisfying $|\nabla{\rm d}|=1$. Up to a rotation of axes, we may assume that $e_N$ is the interior normal to $\partial U$ at $0$, so $\nabla{\rm d}(0)=e_N$. By the regularity of ${\rm d}$ there exist $R_0,C>0$ only depending on $U$ (which will be tacitly assumed henceforth), s.t.\ for all $x\in B_{2R_0}$
\beq\label{bar1}
|\nabla{\rm d}(x)-e_N| \le C|x|, \qquad |{\rm d}(x)-x_N| \le C|x|^2.
\eeq
Also, there exists a $C^{1,1}$-diffeomorphism $\Theta:B_{2R_0}\to\Theta(B_{2R_0})$ (meaning that both $\Theta$ and $\Theta^{-1}$ are $C^{1,1}$) with the following properties: $\Theta(0)=0$, $D\Theta(0)=D\Theta^{-1}(0)={\rm Id}$, and
\[\Theta(B_{2R_0}\cap U) \subset \R^N_+, \qquad \Theta(B_{2R_0}\cap\partial U) \subset \partial\R^N_+,\]
where we have set
\[\R^N_+ = \big\{x'\in\R^N:\,x'_N>0\big\}.\]
Besides, we have
\[\|\Theta\|_{C^{1,1}(B_{2R_0})}+\|\Theta^{-1}\|_{C^{1,1}(\Theta(B_{2R_0}))} \le C,\]
and for all $x\in B_{2R_0}$, $x'\in\Theta(B_{2R_0})$
\[|D\Theta(x)-{\rm Id}|+|D\Theta^{-1}(x')-{\rm Id}| \le CR_0, \qquad |\Theta(x)-x|+|\Theta^{-1}(x')-x'| \le CR_0^2.\]
Set for all $x'\in\Theta(B_{2R_0})$
\[{\rm d}'(x') = {\rm d}(\Theta^{-1}(x')).\]
By \eqref{bar1} and the bounds on $\Theta$, ${\rm d}'\in C^{1,1}(\Theta(B_{2R_0}))$ and for all $x'\in\Theta(B_{2R_0})$ we have
\beq\label{bar2}
|\nabla{\rm d}'(x')-e_N| \le CR_0, \qquad |{\rm d}'(x')-x'_N| \le CR_0^2.
\eeq
Fix $\lambda_0\in(0,1/2)$ (to be determined later). Then set for all $|\lambda|\le\lambda_0$, $R\in (0, R_0/2]$, and $x'\in\Theta(B_{2R_0})$
\[\psi_\lambda(x') = \Big(1+\lambda\varphi\Big(\frac{\Theta^{-1}(x')}{R}\Big)\Big)^\frac{1}{s}.\]
So $\psi_\lambda\in C^{1,1}(\Theta(B_{2R_0}))$ satisfies for all $x'\in\Theta(B_{2R_0})$
\[|\psi_\lambda(x')-1| \le C|\lambda|\chi_{\Theta(B_R)}(x'), \qquad |\nabla\psi_\lambda(x')| \le C\frac{|\lambda|}{R}\chi_{\Theta(B_R)}(x'),\]
and almost everywhere
\[|D^2\psi_\lambda(x')| \le C\frac{\lambda^2}{R^2}\chi_{\Theta(B_R)}(x').\]
Next we define another local diffeomorphism. Set for all $x'\in\Theta(B_{2R_0})$
\[\Psi_\lambda(x') = \big(x'_1,\ldots x'_{N-1},\psi_\lambda(x'){\rm d}'(x')\big),\]
so that $\Psi_\lambda:\Theta(B_{2R_0})\to\R^N$ is a $C^{1,1}$-map s.t.\ $\Psi_\lambda(\Theta(B_{2R_0})\cap\partial\R^N_+)\subset\partial\R^N_+$. We compute the first-order derivatives of $\Psi_\lambda$ at $x'\in\Theta(B_{2R_0})$. Clearly, for all $j\in\{1,\ldots N-1\}$
\[\nabla\Psi_\lambda^j(x') = e_j.\]
For the $N$-th component, noting that $|{\rm d}'|\le CR$ in $\Theta(B_R)$, using \eqref{bar2} and the estimates on $\psi_\lambda$ we have
\begin{align*}
|\nabla\Psi_\lambda^N(x')-e_N| &\le |\nabla\psi_\lambda(x')||{\rm d}'(x')|+|\psi_\lambda(x')-1||\nabla{\rm d'}(x')|+|\nabla{\rm d}'(x')-e_N| \\
&\le C\frac{|\lambda|}{R}|{\rm d}'(x')|\chi_{\Theta(B_R)}(x')+C|\lambda|\chi_{\Theta(B_R)}(x')+CR_0 \\
&\le C|\lambda|\chi_{\Theta(B_R)}(x')+CR_0.
\end{align*}
Taking $\lambda_0$, $R_0$ even smaller if necessary (still depending on $U$), we may assume that $\|D\Psi_\lambda-{\rm Id}\|_{L^\infty(\Theta(B_{2R_0}))}$ is sufficiently small, so that $\Psi_\lambda$ is a $C^{1,1}$-diffeomorphism. A similar argument leads for a.e.\ $x'\in\Theta(B_{2R_0})$ to the following estimate of the second-order derivatives:
\begin{align*}
|D^2\Psi_\lambda(x')| &\le C\Big[\frac{\lambda^2}{R^2}|{\rm d}'(x')|+\frac{|\lambda|}{R}+|\lambda|\Big]\chi_{\Theta(B_R)}(x')+ C \\
&\le C\frac{|\lambda|}{R}\chi_{\Theta(B_R)}(x')+C.
\end{align*}
Now set for all $x\in B_{2R_0}$
\[\Phi_\lambda(x) = \Psi_\lambda(\Theta(x)).\]
By the previous estimates $\Phi_\lambda:B_{2R_0}\to\Phi_\lambda(B_{2R_0})$ is a $C^{1,1}$-diffeomorphism s.t.\ for all $x\in B_{2R_0}$, $\tilde x\in\Phi_\lambda(B_{2R_0})$
\beq\label{bar3}
|D\Phi_\lambda(x')-{\rm Id}|+|D\Phi_\lambda^{-1}(\tilde x)-{\rm Id}| \le C(\lambda_0+R_0),
\eeq
and for a.e.\ $x\in B_{2R_0}$
\beq\label{bar4}
|D^2\Phi_\lambda(x)| \le C\frac{|\lambda|}{R}\chi_{B_R}(x)+C.
\eeq
Also, for all $x\in B_{2R_0}$ a direct computation yields
\beq\label{bar5}
v_\lambda(x) = (\Phi_\lambda^N)_+^s(x).
\eeq
We aim at extending $\Phi_\lambda$ to a {\em global} diffeomorphism, keeping uniform estimates with respect to $\lambda$, $R$. Note that for all $|\lambda|\le\lambda_0$ we have $\Phi_\lambda=\Phi_0$ in $B_{2R_0}\setminus B_{R_0/2}$, while by \eqref{bar1} we have for all $x\in B_{2R_0}\setminus B_{R_0/2}$
\[|\Phi_0(x)-x| \le CR_0^2.\]
So we pick a cut-off function $\eta\in C^\infty_c(B_{2R_0})$ s.t.\ $0\le\eta\le 1$ in $\R^N$, $\eta=1$ in $B_{R_0}$, and 
\[R_0|\nabla\eta|+R_0^2|D^2\eta|\le C(N)\]
in $B_{2R_0}\setminus B_{R_0}$. Then, set for all $x\in\R^N$
\[\hat\Phi_\lambda(x) = \eta(x)\Phi_\lambda(x)+(1-\eta(x))x.\]
Clearly we have $\hat\Phi_\lambda(x)=\Phi_\lambda(x)$ for all $x\in B_{R_0}$, as well as $\hat\Phi_\lambda(x)=x$ for all $x\in B_{2R_0}^c$. Also, for all $x\in B_{2R_0}\setminus B_{R_0/2}$ the previous relations and \eqref{bar3} imply
\begin{align*}
|D\hat\Phi_\lambda(x)-{\rm Id}| &\le |\nabla\eta(x)||\Phi_0(x)-x|+\eta(x)|D\Phi_\lambda(x)-{\rm Id}| \\
&\le \frac{C}{R_0}R_0^2+C(\lambda_0+R_0) \le C(\lambda_0+R_0).
\end{align*}
Moreover, a similar estimate holds for $D\hat\Phi_\lambda^{-1}$ and $D^2\hat\Phi_\lambda$ satisfies \eqref{bar4} a.\,e.. Adjusting $\lambda_0$, $R_0$ again, and setting for simplicity $\hat\Phi_\lambda=\Phi_\lambda$, we deduce that $\Phi_\lambda:\R^N\to\R^N$ is a $C^{1,1}$-diffeomorphism satisfying \eqref{bar3} \eqref{bar4} in $\R^N$, and \eqref{bar5} in $B_{R_0}$, which concludes the first part of our argument.
\vskip2pt
\noindent
Now we aim at proving the bound on $\fpl v_\lambda$ in $D_{R_0/2}=B_{R_0/2}\cap U$, by means of \eqref{bar5} and the properties of $\Phi_\lambda$. Recall that $|\lambda|\le\lambda_0$, $R\in(0,R_0/2]$. Set for all $\eps\in(0,1)$, $x\in B_{R_0/2}\cap U$
\[f_\eps(x) = \int_{\{|\Phi_\lambda(x)-\Phi_\lambda(y)|\ge\eps\}}\frac{(v_\lambda(x)-v_\lambda(y))^{p-1}}{|x-y|^{N+ps}}\,dy\]
(we omit henceforth the dependance on $\lambda$, $R$ for brevity). We claim that for any such $\lambda$ and $R$ there exists $f_0\in L^\infty(D_{R_0/2})$ s.t.
\beq
\label{bclaim}
f_\eps\to f_0\quad \text{in $L^1_{\rm loc}(D_{R_0/2})$ as $\eps\to 0^+$ and}\quad 
\|f_0\|_{L^\infty(D_{R_0/2})} \le C\Big(1+\frac{|\lambda|}{R^s}\Big).
\eeq
First note that it suffices to prove the claim  for
\[\tilde f_\eps(x) = \int_{B_{R_0}\cap\{|\Phi_\lambda(x)-\Phi_\lambda(x)|\ge\eps\}}\frac{(v_\lambda(x)-v_\lambda(y))^{p-1}}{|x-y|^{N+ps}}\,dy.\]
Indeed, for all $\eps\in(0,1)$, $x\in D_{R_0/2}$, $y\in B_{R_0}^c$ we have $|x-y|\ge|y|/2$, hence
\begin{align*}
|f_\eps(x)-\tilde f_\eps(x)| &\le \int_{B_{R_0}^c}\frac{|v_\lambda(x)-v_\lambda(y)|^{p-1}}{|x-y|^{N+ps}}\,dy \\
&\le C\, {\rm diam}(U)^{(p-1)s}\int_{B_{R_0}^c}\frac{dy}{|y|^{N+ps}} \le C,
\end{align*}
for a suitable $C=C(N,p,s,U, \varphi)>0$ (such dependance of $C$ will be assumed henceforth). So, $f_\eps-\tilde f_\eps$ converges in $L^\infty(D_{R_0/2})$ to some universally bounded limit function, i.e., with a bound on the $L^\infty$-norm independent of $\lambda$, $R$. 
\vskip2pt
\noindent
Focusing on $\tilde f_\eps$, we use the change of variables $\tilde x=\Phi_\lambda(x)$, $\tilde y=\Phi_\lambda(y)$ and recall \eqref{bar5} to get 
\[\tilde f_\eps(x) = \int_{\Phi_\lambda(B_{R_0})\cap B_\eps^c(\tilde x)}\frac{((\tilde x_N)_+^s-(\tilde y_N)_+^s)^{p-1}}{|\Phi_\lambda^{-1}(\tilde x)-\Phi_\lambda^{-1}(\tilde y)|^{N+ps}}|{\rm det}\,D\Phi_\lambda^{-1}(\tilde y)|\,d\tilde y\]
for all $x\in D_{R_0/2}$.
We split the integral by setting for all $\tilde x\in\Phi_\lambda(D_{R_0/2})$, $\tilde y\in\Phi_\lambda(B_{R_0})$
\[g_\eps(\tilde x) = \int_{B_\eps^c(\tilde x)}\frac{((\tilde x_N)_+^s-(\tilde y_N)_+^s)^{p-1}}{|D\Phi_\lambda^{-1}(\tilde x)(\tilde x-\tilde y)|^{N+ps}}\,d\tilde y,\]
\[h_\eps(\tilde x) = \int_{\Phi_\lambda^c(B_{R_0})\cap B_\eps^c(\tilde x)}\frac{((\tilde x_N)_+^s-(\tilde y_N)_+^s)^{p-1}}{|D\Phi_\lambda^{-1}(\tilde x)(\tilde x-\tilde y)|^{N+ps}}\,d\tilde y,\]
and
\[H(\tilde x,\tilde y) = |{\rm det}\,D\Phi_\lambda^{-1}(\tilde y)|-|{\rm det}\,D\Phi_\lambda^{-1}(\tilde x)|\frac{|\Phi_\lambda^{-1}(\tilde x)-\Phi_\lambda^{-1}(\tilde y)|^{N+ps}}{|D\Phi_\lambda^{-1}(\tilde x)(\tilde x-\tilde y)|^{N+ps}}.\]
Indeed, a direct computation shows for all $x\in D_{R_0/2}$, $\tilde x=\Phi_\lambda(x)$
\[\tilde f_\eps(x) = |{\rm det}\,D\Phi_\lambda^{-1}(\tilde x)|(g_\eps(\tilde x)-h_\eps(\tilde x))+\int_{\Phi_\lambda(B_{R_0})\cap B_\eps^c(\tilde x)}\frac{((\tilde x_N)_+^s-(\tilde y_N)_+^s)^{p-1}}{|\Phi_\lambda^{-1}(\tilde x)-\Phi_\lambda^{-1}(\tilde y)|^{N+ps}}H(\tilde x,\tilde y)\,d\tilde y.\]
By \cite[Lemma 3.2]{IMS} we have $g_\eps\to 0$ in $L^\infty_{\rm loc}(\R^N_+)$ as $\eps\to 0^+$. Also, a similar argument to that used above for $f_\eps-\tilde f_\eps$ shows that $h_\eps$ converges in $L^\infty(\Phi_\lambda(D_{R_0/2}))$ to a function with a universal $L^\infty$-bound.
\vskip2pt
\noindent
So we turn to the last quantity. We claim that there exists $C=C(N,p,s,U, \varphi)>0$ s.t.\ for all $|\lambda|\le\lambda_0$, $R\in(0,R_0/2]$ we have for all $\tilde x\in\Phi_\lambda(D_{R_0/2})$
\beq\label{bar6}
\int_{\Phi_\lambda(B_{R_0})}\frac{|(\tilde x_N)_+^s-(\tilde y_N)_+^s|^{p-1}}{|\Phi_\lambda^{-1}(\tilde x)-\Phi_\lambda^{-1}(\tilde y)|^{N+ps}}|H(\tilde x,\tilde y)|\,d\tilde y \le C\Big(1+\frac{|\lambda|}{R^s}\Big).
\eeq
We go back to the original variables using \eqref{bar5}:
\[\int_{\Phi_\lambda(B_{R_0})}\frac{|(\tilde x_N)_+^s-(\tilde y_N)_+^s|^{p-1}}{|\Phi_\lambda^{-1}(\tilde x)-\Phi_\lambda^{-1}(\tilde y)|^{N+ps}}|H(\tilde x,\tilde y)|\,d\tilde y = \int_{B_{R_0}}\frac{|v_\lambda(x)-v_\lambda(y)|^{p-1}}{|x-y|^{N+ps}}|K(x,y)|\,dy,\]
where we have set for all $x\in D_{R_0/2}$, $y\in\R^N$
\[K(x,y) = 1-\frac{|{\rm det}\,D\Phi_\lambda(y)|}{|{\rm det}\,D\Phi_\lambda(x)|}\,\frac{|x-y|^{N+ps}}{|D\Phi_\lambda(x)^{-1}(\Phi_\lambda(x)-\Phi_\lambda(y))|^{N+ps}}.\]
We estimate $K$ arguing as in \cite[Lemma 3.4]{IMS1}. First we split and use \eqref{bar3}:
\begin{align*}
|K(x,y)| &\le \left|\frac{|{\rm det}\,D\Phi_\lambda(x)|-|{\rm det}\,D\Phi_\lambda(y)|}{|{\rm det}\,D\Phi_\lambda(x)|} \right|\\
& \ + \frac{|{\rm det}\,D\Phi_\lambda(y)|}{|{\rm det}\,D\Phi_\lambda(x)|}\left|1-\frac{|x-y|^{N+ps}}{|D\Phi_\lambda(x)^{-1}(\Phi_\lambda(x)-\Phi_\lambda(y))|^{N+ps}}\right| \\
&\le C\left|{\rm det}\,D\Phi_\lambda(x)-{\rm det}\,D\Phi_\lambda(y)\right|+C\left|1-\frac{|x-y|^{N+ps}}{|D\Phi_\lambda(x)^{-1}(\Phi_\lambda(x)-\Phi_\lambda(y))|^{N+ps}}\right| \\
&= K_1(x,y)+K_2(x,y).
\end{align*}
Focusing on   $K_1$, we use \eqref{bar3} and \eqref{bar4} to get
\begin{align*}
K_1(x,y) &\le C\left|\int_0^1\frac{d}{dt}{\rm det}\,D\Phi_\lambda(x+t(y-x))\,dt\right| \\
&\le C\int_0^1|D^2\Phi_\lambda(x+t(y-x))||x-y|\,dt \\
&\le C\, |x-y|\int_0^1\Big[\frac{|\lambda|}{R}\chi_{B_R}(x+t(y-x))+1\Big]\,dt \\
&\le C\, |\lambda|\min\Big\{\frac{|x-y|}{R},\,1\Big\}+C\,|x-y|.
\end{align*}
Considering now $K_2$, using Taylor's expansion with integral remainder, \eqref{bar3} and \eqref{bar4} we obtain a formally analogous estimate:
\begin{align*}
K_2(x,y) &\le C\left|1-\frac{|x-y|^2}{|D\Phi_\lambda(x)^{-1}(\Phi_\lambda(x)-\Phi_\lambda(y))|^2}\right| \\
& \le C\,\frac{\big|D\Phi_\lambda(x)^{-1}(\Phi_\lambda(x)-\Phi_\lambda(y))+(x-y)\big|\,\big|D\Phi_\lambda(x)^{-1}(\Phi_\lambda(x)-\Phi_\lambda(y))-(x-y)\big|}{|D\Phi_\lambda(x)^{-1}(\Phi_\lambda(x)-\Phi_\lambda(y))|^2} \\
&\le C\,\frac{|\Phi_\lambda(x)-\Phi_\lambda(y)|+|D\Phi_\lambda(x)(x-y)|}{|\Phi_\lambda(x)-\Phi_\lambda(y)|^2}\, \big|\Phi_\lambda(x)-\Phi_\lambda(y)-D\Phi_\lambda(x)(x-y)| \\
&\le \frac{C}{|x-y|}\left|\int_0^1(1-t)\frac{d^2}{dt^2}\Phi_\lambda(x+t(y-x))\,dt\right| \\
&\le \frac{C}{|x-y|}\int_0^1\Big[\frac{|\lambda|}{R}\chi_{B_R}(x+t(y-x))+1\Big]|x-y|^2\,dt \\
&\le C\, |\lambda|\min\Big\{\frac{|x-y|}{R},\,1\Big\}+C\, |x-y|.
\end{align*}
Therefore, for all $x\in D_{R_0/2}$, $y\in\R^N$ we have
\[|K(x,y)| \le C\, |\lambda|\min\Big\{\frac{|x-y|}{R},\,1\Big\}+C\, |x-y|,\]
with $C=C(N,p,s,U, \varphi)>0$. Moreover, from \eqref{bar3} we infer  a Lipschitz bound on $\Phi_\lambda$ independent on $\lambda$ and $R$ as long as $|\lambda|\le \lambda_0$ and $R\in(0,R_0/2]$, hence \eqref{bar5} shows that $v_\lambda$ has $s$-H\"older modulus of continuity  in $B_{R_0}$ uniformly bounded in $\lambda$ and $R$ under these conditions. Therefore
\[
\frac{|v_\lambda(x)-v_\lambda(y)|^{p-1}}{|x-y|^{N+ps}}\le C\frac{|x-y|^{s(p-1)}}{|x-y|^{N+ps}}=\frac{C}{|x-y|^{N+s}}.
\]
 Plugging these estimates into the integral we have
\begin{align*}
\int_{B_{R_0}}&\frac{|v_\lambda(x)-v_\lambda(y)|^{p-1}}{|x-y|^{N+ps}}|K(x,y)|\,dy \\
&\le \int_{B_{R_0}}\frac{C}{|x-y|^{N+s}}\Big[|\lambda|\min\Big\{\frac{|x-y|}{R},\,1\Big\}+|x-y|\Big]\,dy \\
&\le C\, \frac{|\lambda|}{R}\int_{B_{R_0}}\frac{\min\{|x-y|,\,R\}}{|x-y|^{N+s}}\,dy+C\int_{B_{R_0}}\frac{dy}{|x-y|^{N+s-1}} \\
&\le C\, \frac{|\lambda|}{R}\int_{B_R(x)}\frac{dy}{|x-y|^{N+s-1}}+C\, |\lambda|\int_{B_{R_0}\cap B_R^c(x)}\frac{dy}{|x-y|^{N+s}}+C \\
&\le C\, \frac{|\lambda|}{R^s}+C\, |\lambda|\int_{B_R^c(x)}\frac{dy}{|x-y|^{N+s}}+C \le C\frac{|\lambda|}{R^s}+C.
\end{align*}
Thus we have proved \eqref{bar6}. In turn, that implies \eqref{bclaim} (the stated convergence is actually in $L^\infty_{\rm loc}(D_{R_0/2})$ as $\eps\to 0^+$). 
\vskip2pt
\noindent
Finally, \cite[Lemma 2.5]{IMS} can be applied to $v_\lambda$ in $D_{R_0/2}$ with the sets
\[A_\eps=\left\{(x, y)\in \R^N\times \R^N: |\Phi_\lambda(x)-\Phi_\lambda(y)|<\eps\right\}\]
thanks to  the global Lipschitzianity of $\Phi_\lambda^{-1}$ granted by \eqref{bar3}. Therefore $\fpl v_\lambda=f_0$ weakly in $D_{R_0/2}$ and the   bound on $f_0$ in \eqref{bclaim} concludes the proof of the Proposition with $\rho'_U=R_0/2$ and $\lambda_0$ as before. 
\qed
\vskip4pt
\noindent
{\bf Acknowledgement.} Both authors are members of GNAMPA (Gruppo Nazionale per l'Analisi Matematica, la Probabilit\`a e le loro Applicazioni) of INdAM (Istituto Nazionale di Alta Matematica 'Francesco Severi') and are supported by GNAMPA project {\em Problemi non locali di tipo stazionario ed evolutivo} (CUP E53C23001670001). S.M.\ is partially supported by the projects PIACERI linea 2 and linea 3 of the University of Catania, GNAMPA project {\em Problemi ellittici e parabolici con termini di reazione singolari e convettivi} and PRIN project 2022ZXZTN2. This work was partially accomplished during S.\ Mosconi's visit to the University of Cagliari, funded by the research project {\em Analysis of PDE's in connection with real phenomena} (CUP F73C22001130007, Fondazione di Sardegna 2021). The authors express their gratitude to E.\ Lindgren for bringing to their attention the interesting reference \cite{GL}.

\end{document}